\documentclass[a4paper, 10pt]{article}

\usepackage{graphics,graphicx}
\usepackage{amssymb,amsmath}
\usepackage{fullpage}
\usepackage{enumitem}
\usepackage{xcolor}
\usepackage{thmtools}

\usepackage[capitalize]{cleveref}

\newtheorem{theorem}{Theorem}

\newtheorem{claim}{Claim}

\newtheorem{claim2} {Claim}

\newtheorem{cor}{Corollary}
 \newtheorem{lemma}{Lemma}

\def \sm {\setminus}
\def \es {\emptyset}

\newenvironment{proof}[1][]%
{\noindent {\setcounter{equation}{0}\it Proof.
	}{#1}{}}{\hfill$\sq$\vspace{0ex}}

\newenvironment{proofthm}[1][]%
{\noindent {\setcounter{equation}{0}\it Proof.
	}{#1}{}}{\hfill$\Box$\vspace{2ex}}

\newcommand*\sq{\mathbin{\vcenter{\hbox{\rule{0.75ex}{1.0ex}}}}}

\begin{document}
	\title{An optimal chromatic bound  for  ($P_2+P_3$, gem)-free graphs}
	\author{ Arnab Char\thanks{Computer Science Unit, Indian Statistical
	Institute, Chennai Centre, Chennai 600029, India. } \and T.~Karthick\thanks{Corresponding author, Computer Science Unit, Indian Statistical
Institute, Chennai Centre, Chennai 600029, India. Email: karthick@isichennai.res.in. ORCID: 0000-0002-5950-9093. This research is partially supported by National Board of Higher Mathematics (NBHM), DAE, India. }}
	
	\date{\today}
	
	\maketitle

	\begin{abstract}
	Given a graph $G$, the parameters $\chi(G)$ and $\omega(G)$ respectively denote the chromatic number and the clique number of $G$. A  function $f : \mathbb{N} \rightarrow \mathbb{N}$  such that $f(1) = 1$ and $f(x) \geq x$, for all $x \in \mathbb{N}$  is called a {\it $\chi$-binding function}
	for the given  class of graphs $\cal{G}$ if every $G \in \cal{G}$ satisfies $\chi(G) \leq f(\omega(G))$, and the  \emph{smallest $\chi$-binding function} $f^*$ for $\cal{G}$ is defined as $f^*(x) := \max\{\chi(G)\mid G\in {\cal G} \mbox{ and } \omega(G)=x\}$.   In general, the problem of obtaining the smallest $\chi$-binding function for the given class of graphs seems to be extremely hard, and only a few classes of graphs are studied in this direction.  In this paper, we study the class of ($P_2+ P_3$, gem)-free graphs, and prove that the function $\phi:\mathbb{N}\rightarrow \mathbb{N}$ defined by $\phi(1)=1$,  $\phi(2)=4$, $\phi(3)=6$   and  $\phi(x)=\left\lceil\frac{1}{4}(5x-1)\right\rceil$, for  $x\geq 4$
		is the smallest $\chi$-binding function for the class of ($P_2+ P_3$, gem)-free graphs.
		\end{abstract}


\section{Introduction}
	Given a graph $G$, the parameters $\chi(G)$ and $\omega(G)$ respectively denote the chromatic number and the clique number of $G$.  A graph $G$ is \emph{perfect} if every induced subgraph $H$ of $G$ satisfies $\chi(H)=\omega(H)$.
	If $\cal G$ is a hereditary class of graphs, for each $G\in \cal G$, while $\omega(G)$ is an obvious lower bound for $\chi(G)$,  it is an interesting problem to find an upper bound  for $\chi(G)$ in terms of $\omega(G)$.  The existence of graphs with large chromatic number and small clique number motivated Gy\'arf\'as \cite{Gyarfas} to initiate the systematic study of the notion of  $\chi$-boundedness  for the given (hereditary) class of graphs. A  function $f : \mathbb{N} \rightarrow \mathbb{N}$  such that $f(1) = 1$ and $f(x) \geq x$, for all $x \in \mathbb{N}$  is called a {\it $\chi$-binding function}
	for the given (hereditary) class of graphs $\cal{G}$ if every $G \in \cal{G}$ satisfies $\chi(G) \leq f(\omega(G))$. A graph class $\cal G$ is said to be \emph{$\chi$-bounded} if it has a  $\chi$-binding function.    The \emph{smallest $\chi$-binding function} $f^*$ for $\cal{G}$ is defined as $f^*(x) := \max\{\chi(G)\mid G\in {\cal G} \mbox{ and } \omega(G)=x\}$.
In this paper, we are interested in the smallest $\chi$-binding function  for the  given class  of graphs, and it  is studied only for a few classes of graphs in the literature. For instance, the function $f(x) = x$ is the smallest $\chi$-binding function for the class of perfect graphs, and we refer to \cite{Gyarfas,SR-Poly-Survey} for a few more   classes of graphs.

%

As usual (for a positive integer $t$), $P_t$, $C_t$, and $K_t$ respectively denote the induced path, induced cycle, and the
complete graph on $t$ vertices. We say that a graph $G$ \emph{contains} a graph $H$ if $H$ is an induced subgraph of $G$.  A graph is \emph{$H$-free} if it does not contain $H$. A graph is ($H_1,H_2, \ldots, H_k$)-free if it does not contain any graph in $\{H_1,H_2,\ldots,H_k\}$.
  For two vertex disjoint graphs $G_1$ and $G_2$, their {\it union } $G_1+G_2$ is a graph with vertex-set $V(G_1)\cup V(G_2)$ and
the edge-set $E(G_1)\cup E(G_2)$.  The graph $\ell G$ denotes the union of $\ell$ copies of the same graph $G$; for instance, $2K_2$ is the graph  $K_2+K_2$. An   \emph{odd hole} is the graph $C_{2t+1}$, where $t\geq 2$. A \emph{diamond} is the complement graph of  a $P_1+P_3$, a \emph{gem} is the complement graph of a $P_1+P_4$, and  the \emph{Co-Schl\"afli graph}  is the complement graph of   the 16-regular Schl\"afli graph on 27 vertices.

The class of ($P_2+P_3$)-free graphs has attracted many researchers since it includes two  intriguing and well-studied graph    classes, namely the class of $3K_1$-free graphs and the class of $2K_2$-free graphs; see \cite{SR-Poly-Survey}.   It is also an interesting subclass of   the class of $P_6$-free graphs. It is  known that $f(x)=\binom{x+2}{3}$ is a  $\chi$-binding function for the class of ($P_2+P_3$)-free graphs \cite{BC}, and that the class of ($P_2+P_3$)-free graphs  does not admit a linear $\chi$-binding function \cite{Gyarfas}. However  the problem of obtaining the smallest $\chi$-binding function for the class of ($P_2+P_3$)-free graphs is open, and seems to be hard.  So we focus on the smallest $\chi$-binding functions for some subclasses of ($P_2+P_3$)-free graphs. Recently,   we \cite{AK3}  proved  that   the function $g:{\mathbb{N}}\rightarrow {\mathbb{N}}$   defined by $g(1)=1$, $g(2)=4$, and $g(x)= \max\{x+3,\lfloor \frac{3}{2}x\rfloor-1\}$, for  $x\geq 3$ is the smallest $\chi$-binding  function  for the class of ($P_2+P_3$,  $\overline{P_2+ P_3}$)-free graphs.

	From a more general result of Schiermeyer and Randerath \cite{SR-Poly-Survey}, we see that every ($P_2+P_3$, gem)-free graph $G$ satisfies $\chi(G)\leq 4(\omega(G)-1)$. In 2023, Prashant et al \cite{PFG} improved this bound and  showed that every  ($P_2+P_3$, gem)-free graph $G$ satisfies $\chi(G)\leq 2\,\omega(G)$.  In this paper, we obtain the   smallest $\chi$-binding  function  for the class of ($P_2+P_3$, gem)-free graphs.

To state our results, we need the following. Let $G$ be a given   graph on $n$ vertices, say $v_1, v_2, \ldots , v_n$, and let
$K_{t_1}, K_{t_2}, \ldots, K_{t_n}$ be   vertex-disjoint complete graphs. Then a \emph{clique blowup}
  of $G$, denoted by $G\langle K_{t_1}, K_{t_2}, \ldots, K_{t_n}\rangle$,  is the graph
obtained from $G$ by (a) replacing each vertex $v_i$ of $G$ by $K_{t_i}$, and
(b) for all  $i, j\in \{1,2, \ldots, n\}$ and $i\neq j$,  if $v_i$ and
$v_j$ are adjacent (resp. nonadjacent) in $G$, then each vertex in $V(K_{t_i})$ is adjacent to (resp. nonadjacent) to every vertex in $V(K_{t_j})$.
A classical and a celebrated result of Lov\'asz \cite{Lovasz-1} states that any clique blowup of a perfect graph is a perfect graph.

Let
	${\cal C}$ denote the class of graphs, where every $G\in \cal C$ is isomorphic to the union of an $\ell K_1$ and a      clique blowup of   $C_5$, for some $\ell\in \mathbb{N}\cup \{0\}$.   Clearly every graph in ${\cal C}$ is $(P_2+P_3$, gem$)$-free, and it is  known that \cite{Geiber} every $G \in \cal C$ satisfies $\chi(G)\leq \left\lceil\frac{5\omega(G)-1}{4}\right\rceil$. Also from a result of Randerath et al \cite{Randerath-Tewes} it is known that  every ($P_2+ P_3$, $K_3$)-free graph $G$ satisfies  $\chi(G)\leq 4$. Here we prove the following.

\begin{theorem}\label{thm:p2p3K4-bnd}
 Every ($P_2+ P_3$, gem)-free graph  $G$ with $\omega(G)=3$ satisfies  $\chi(G)\leq 6$.
\end{theorem}

\begin{theorem}\label{structhm}
			Let $G$ be a ($P_2+ P_3$, gem)-free graph with $\omega(G)\geq 4$. Then either $G\in {\cal C}$ or $\chi(G)\leq \omega(G)+1$.
					\end{theorem}

(The proofs of \cref{thm:p2p3K4-bnd} and \cref{structhm} are given in \cref{sec:K4} and \cref{sec:C5} respectively.) From \cref{structhm}, we immediately have the following.

\begin{cor}\label{thm:p2p3-bnd}
 Every ($P_2+ P_3$, gem)-free graph $G$ with $\omega(G)\geq 4$ satisfies  $\chi(G)\leq  \left\lceil\frac{5\omega(G)-1}{4}\right\rceil.$ \end{cor}

\begin{proofthm}
 Let $G$ be a ($P_2+ P_3$, gem)-free graph with $\omega(G)\geq 4$. Then since $\omega(G)+1 \leq \left\lceil\frac{5\omega(G)-1}{4}\right\rceil$, and since every $G\in \cal C$ satisfies $\chi(G)\leq \lceil\frac{5\omega(G)-1}{4}\rceil$, the proof follows from \cref{structhm}.
\end{proofthm}

\begin{cor}\label{thm:p2p3-bndopt}
\label{optchrombd}
		The function $\phi:\mathbb{N}\rightarrow \mathbb{N}$ defined by $$\phi(1)=1,~ \phi(2)=4, ~\phi(3)=6 \mbox{ and } \phi(x)=\left\lceil\frac{5x-1}{4}\right\rceil, \mbox{ for }  x\geq 4$$
		is the smallest $\chi$-binding function for the class of ($P_2+ P_3$, gem)-free graphs.
\end{cor}
\begin{proofthm}
Let $\cal P$ be the class of ($P_2+ P_3$, gem)-free graphs.  Then from an earlier stated result of Randerath et al  \cite{Randerath-Tewes},  \cref{thm:p2p3K4-bnd} and from \cref{thm:p2p3-bnd},   clearly the function $\phi$ is a $\chi$-binding function for $\cal P$.  To prove that $\phi$ is the smallest $\chi$-binding function for $\cal P$, we prove that for each $k\in \mathbb{N}$, there is a   graph $G_k \in \cal P$ such that $\omega(G_k)=k$ and $\chi(G_k) =\phi(\omega(G_k))$, and they are given below : For $k\in \{1,2,3\}$, clearly $G_1\cong \ell K_1$, $G_2\cong$ Gr\"otzsch graph/Mycielski's 4-chromatic triangle-free graph, and  $G_3\cong$ Co-Schl\"afli graph are our required graphs.  For $k\geq 4$, consider the graphs given in \cref{tab} (where the graph $C_5$ is taken with vertex-set $\{v_1,v_2,v_3,v_4,v_5\}$ and   edge-set  $\{v_1v_2,v_2v_3,v_3v_4,v_4v_5,v_5v_1\}$, and $t\in \mathbb{N}$).

\begin{table}[h]
\centering
\begin{tabular}{|c|c|c|}
\hline
$k$ & $G_k$ & $\chi(G_k)$\\
 \hline
$4t$& $C_5\langle K_{2t}, K_{2t}, K_{2t}, K_{2t}, K_{2t}\rangle$ & $5t$ \\
\hline
$4t+1$& $C_5\langle K_{2t+1}, K_{2t}, K_{2t+1}, K_{2t}, K_{2t}\rangle$ & $5t+1$ \\
\hline
$4t+2$& $C_5\langle K_{2t+1}, K_{2t+1}, K_{2t+1}, K_{2t+1}, K_{2t+1}\rangle$ & $5t+3$ \\
\hline
$4t+3$& $C_5\langle K_{2t+2}, K_{2t+1}, K_{2t+2}, K_{2t+1}, K_{2t+1}\rangle$ & $5t+4$ \\
\hline
\end{tabular}
\caption{Some extremal graphs.} \label{tab}
\end{table}
\noindent{}Clearly for $k\geq 4$, each graph $G_k$ given in \cref{tab} is   in ${\cal P} \cap {\cal C}$ with $\omega(G_k)=k$, and hence $\chi(G)\leq \lceil \frac{5k-1}{4}\rceil$. Also since  the stability number of $G_k$, $\alpha(G_k) =2$ and since $\chi(G_k) \geq \frac{|V(G_k)|}{\alpha(G_k)}$, we have $\chi(G_k)\geq \lceil \frac{5k-1}{4}\rceil$. This completes the proof of \cref{optchrombd}.
\end{proofthm}

Our next theorem completely characterizes
 the class of  ($P_2+ P_3$, gem)-free graphs ${\cal P}^*$ where every  $G\in {\cal P}^*$  satisfies  $\chi(G)=\left\lceil \frac{5\omega(G)-1}{4}\right\rceil$.  For $t\in \mathbb{N}$ and $t\geq 4$, we let
	${\cal C}_t:=\{G\in {\cal C}\mid \omega(G)=t \mbox{ and } \chi(G)=\left\lceil \frac{5t-1}{4}\right\rceil\}$.

\begin{theorem}\label{optattaingr}
	$(i)$ Let $G$ is  a ($P_2+ P_3$, gem)-free graph with $\omega(G)\geq 6$. Then $\chi(G)=\left\lceil \frac{5\omega(G)-1}{4}\right\rceil$ if and only if ~$G\in {\cal C}_{\omega(G)}$.  $(ii)$ There is a ($P_2+ P_3$, gem)-free graph $H$ with $\omega(H)\in\{4,5\}$ and $\chi(H)=\left\lceil \frac{5\omega(H)-1}{4}\right\rceil$ and $H\notin {\cal C}$.
\end{theorem}
\begin{proofthm}$(i)$:~Let $G$ be  a ($P_2+ P_3$, gem)-free graph with $\omega(G)\geq 6$. If $\chi(G)=\left\lceil \frac{5\omega(G)-1}{4}\right\rceil$, then
since   $\omega(G)+1<\left\lceil\frac{5\omega(G)-1}{4}\right\rceil$, from \cref{structhm}, it follows that $G\in {\cal C}\cap {\cal C}_{\omega(G)}$. Hence   $(i)$ holds from the definition of ${\cal C}_{\omega(G)}$.

To prove $(ii)$, we first note that for $\omega(G)\in \{4,5\}$, clearly $\omega(G)+1=\left\lceil\frac{5\omega(G)-1}{4}\right\rceil$.  Let $G^*$ be the graph with $V(G^*):=\{u_1,u_2,u_3,u_4,u_5,a\}$ and $E(G^*):=\{u_1u_2,u_2u_3,u_3u_4,u_4u_5,u_5u_1, au_2\}$. For $t\in \{3,4\}$, consider the graph $H\cong G^*\langle K_{t-1}, K_1,  K_{t-1}, K_1, K_1, K_1\rangle$. Clearly $H$ is a ($P_2+ P_3$, gem)-free graph, $H\notin {\cal C}$, and it is easy verify that $\omega(H) = t$  and $\chi(H)=t+1$.
\end{proofthm}	

 The organization of the paper is as follows: We give some  notation, terminology and some preliminaries which we have  used in this paper at the end of this section. In \cref{sec:K4}, we give a proof of \cref{thm:p2p3K4-bnd}. Next we give some useful structural properties of ($P_2+P_3$, gem)-free graphs that contain  a $C_5$ in \cref{genprop}. The rest of the paper is devoted to the proof of \cref{structhm} which is based on a sequence of results obtained from forbidding some special graphs, namely $F_1$, $F_2$ and $F_3$ (see Figure~\ref{fig-F123}), and it is  given at the end of \cref{sec:C5}.

\smallskip
All graphs considered in this paper are without loops and multiple edges, finite and undirected. For notation  and terminology which are not defined here,  we refer to West \cite{west}.   For a graph $G =(V, E)$, let $V(G)$   denote its vertex-set and $E(G)$ its edge-set. Let $\overline{G}$ denote the complement graph of  $G$. A \emph{component} is a maximal connected subgraph of $G$. A \emph{big  component} of $G$ is a  component of $G$ with at least two vertices.
 For a vertex $v \in V(G)$, any vertex which is adjacent (resp. nonadjacent) to $v$  is called a  \emph{neighbor} (resp. \emph{nonneighbor}) of $v$, and  the {\it neighborhood} of
$v$, denoted by $N(v)$,  is the set of neighbors of $v$.     For $S \subseteq V(G)$, let $G[S]$ denote the subgraph  induced by $S$ in $G$. We say that a subset $S$ of $V(G)$ induces a graph $H$ if $G[S]$ is isomorphic to  $H$. For any two
subsets $S$, $T\subseteq V(G)$, we say that $S$ is \emph{complete} to
$T$ if every vertex in $S$ is adjacent to every vertex in $T$, and we say
that $S$ is \emph{anticomplete} to $T$ if there is no edge in $G$ with one end in $S$ and the other in $T$.    A \emph{clique} (resp. {\it  stable  set}) in a graph $G$ is a set of
mutually adjacent (resp. nonadjacent) vertices in $G$. We say that a graph $G$ contains a clique of size $\ell$ ($\leq \omega(G)$)  if there is a clique of size $\ell$  in $G$.

We will also use the following simple observations often: $(a)$  If a graph $G$ is $P_3$-free, then it is the union of vertex-disjoint complete graphs, and hence perfect. $(b)$  If   a graph $G$ is gem-free, then for each vertex $v\in V(G)$, $G[N(v)]$ is a $P_4$-free graph, and since it is long-known that every $P_4$-free graph is perfect, $G[N(v)]$ is perfect.

	\begin{figure}[t]
		\centering
		\includegraphics[width=15cm]{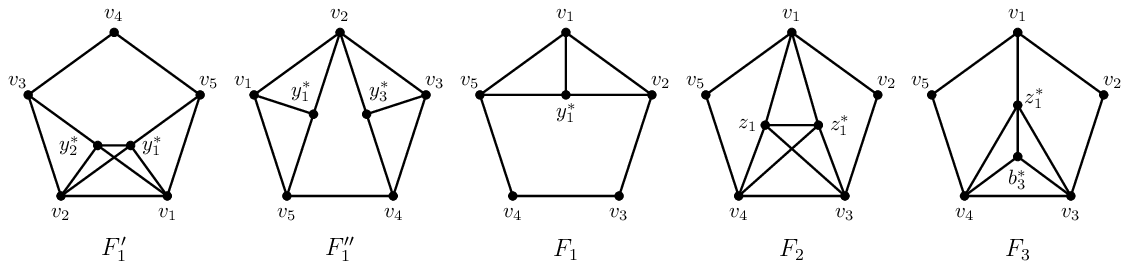}
		\caption{Some special graphs.}\label{fig-F123}
	\end{figure}

	\section{Proof of \cref{thm:p2p3K4-bnd}}\label{sec:K4}

		Let $G$ be a ($P_2+P_3$, gem)-free graph with $\omega(G)=3$. First if $G$ is diamond-free, then from a result of Cameron et al \cite{CHM}, since every ($P_2+P_3$, diamond)-free graph $H$ with $\omega(H)=3$ satisfies $\chi(H)\leq 6$, we have $\chi(G)\leq 6$ and we are done. So we may assume that $G$ contains a diamond, say $D$, and let $V(D):=\{v_1,v_2,v_3,x\}$ and $E(D):= \{v_1v_2,v_2v_3,v_3v_1,xv_2,xv_3\}$. Let $C:=\{v_1,v_2,v_3\}$.
		We let $R:=\{v\in V(G)\sm C \mid N(v)\cap C=\{v_1\} \mbox{ or } N(v)\cap C=\es\}$. Then since $\{v_2,v_3\}$ is anticomplete to $R$ and since $G$ is ($P_2+P_3$)-free, we observe that $G[R]$ is $P_3$-free, and hence perfect; so $\chi(G[R])=\omega(G[R])\leq \omega(G) = 3$. Thus it is enough to prove that $\chi(G[V(G)\sm R])\leq 3$. Since $\omega(G)=3$, we may assume that $G$ is $K_4$-free. To proceed further,
		for $i\in \{1,2,3\}$, $i$ mod $3$, we let: $$S_i:=\{v\in V(G)\sm (R\cup C)\mid uv_i\notin E(G)  \mbox{ and } uv_{i+1}\in E(G)\}.$$
		So $x\in S_1$ and $\{v_1,v_2,v_3\}\subset V(G)\sm R$. Next we have the following:
		\begin{claim2}\label{K4freeVpart}
			$V(G)\sm (R\cup C)=S_1\cup S_2\cup S_3$.
		\end{claim2}
		\begin{proof}Let $u\in V(G)\sm (R\cup C)$. By the definition of $R$,  $u$ has a neighbor in $C$. Since $\{u,v_1,v_2,v_3\}$ does not induce a $K_4$, we observe that $u$ has a nonneighbor in $C$. Now if $uv_2\notin E(G)$,  then by the definition of $R$, we have $uv_3\in E(G)$, and   hence $u\in S_2$. So we may assume that $uv_2\in E(G)$. If $uv_1\notin E(G)$, then $u\in S_1$, otherwise  $u\in S_3$.
		This proves \cref{K4freeVpart}. \end{proof}
		\begin{claim2}\label{K4freest12}
			$S_1$ and $S_2$ are  stable sets.
		\end{claim2}
	\begin{proof}Suppose to the contrary that there are adjacent  vertices in $S_1$,  say $u_1$ and $u_2$. Then since $\{u_1,u_2,v_2,v_3\}$ does not induce a $K_4$, we have $u_1\neq x$, $u_2\neq x$, and we may assume that $u_1v_3\notin E(G)$. Then since $\{u_1,u_2,v_3,v_1,v_2\}$ does not induce a gem, we have $u_2v_3\notin E(G)$. Likewise  $u_1x,u_2x\notin E(G)$. But then $\{u_1,u_2,x,v_3,v_1\}$ induces a $P_2+P_3$ which is a contradiction. So $S_1$ is a stable set. A similar proof shows that $S_2$ is also a stable set.
		This proves \cref{K4freest12}. \end{proof}
		
		\begin{claim2}\label{K4freest3}
			$S_3$ is a stable set.
		\end{claim2}
		\begin{proof}Suppose to the contrary that there are adjacent  vertices in $S_3$,  say $u_1$ and $u_2$. Then by the definition of $R$, we have $u_2v_2, u_1v_2\in E(G)$, and then $\{u_1,u_2,v_1,v_2\}$ induces a $K_4$ which is a contradiction. So $S_3$ is a stable set.
		This proves \cref{K4freest3}. 	\end{proof}
		
		\medskip
		Now since $S_i$ is anticomplete to $\{v_i\}$, from above arguments, we  see that $S_1\cup \{v_1\}$, $S_2 \cup \{v_2\}$ and $S_3\cup \{v_3\}$ are three stable sets.  So from \cref{K4freeVpart}, we have $\chi(G[V(G)\sm R])=3$ and hence $\chi(G)\leq \chi(G[R])+\chi(G[V(G)\sm R])\leq 6$. This completes the proof of \cref{thm:p2p3K4-bnd}. \hfill{$\Box$}

	\section{Properties of $(P_2+P_3$, gem$)$-free graphs that contain a $C_5$}\label{genprop}
	Let $G$ be a  $(P_2+P_3$, gem$)$-free graph. Suppose  that $G$ contains a $C_5$, say with vertex-set $C:=\{v_1,v_2,v_3,v_4,v_5\}$ and edge-set
  $\{v_1v_2,v_2v_3,v_3v_4,v_4v_5,v_5v_1\}$. Then   we have the following.
	\begin{lemma}\label{Vpart}For any  vertex $u \in V(G)\sm C$, we have  $|N(u)\cap C|\leq 3$.
	\end{lemma}
	\begin{proofthm}Suppose not. Then there is an index $i\in \{1,2,3,4,5\}$, $i$ mod $5$ such that $\{v_i,v_{i+1},v_{i+2},v_{i+3}\}\subseteq N(u)$, and then $\{u,v_i,v_{i+1},$ $v_{i+2},v_{i+3}\}$ induces a gem. So \cref{Vpart} holds.
	\end{proofthm}

	For $i\in \{1,2,3,4,5\}$ and $i$ mod $5$, we let:
\begin{center}
  \begin{tabular}{l}
		$A_i:=$ $\{u\in V(G)\sm C \mid N(u)\cap C=\{v_i\}\},$ \\
		$B_i:=$ $\{u\in V(G)\sm C \mid N(u)\cap C= \{v_i,v_{i+1}\}\},$\\
		$X_i:=$ $\{u\in V(G)\sm C \mid N(u)\cap C= \{v_{i-1},v_{i+1}\}\},$\\
		$Y_i:=$ $\{u\in V(G)\sm C \mid N(u)\cap C=  \{v_{i-1},v_i,v_{i+1}\}\},$\\
		$Z_i:=$ $\{u\in V(G)\sm C \mid N(u)\cap C= \{v_{i-2},v_i,v_{i+2}\}\},$  and \\
		$~T:=$ $\{u\in V(G)\sm C \mid N(u)\cap C= \es\}.$
	\end{tabular}
\end{center}
	We let $A:=\cup_{i=1}^5A_i$, $B:=\cup_{i=1}^5B_i$, $X:=\cup_{i=1}^5X_i$, $Y:=\cup_{i=1}^5Y_i$  and $Z:=\cup_{i=1}^5Z_i$. Then clearly $V(G):= C\cup A\cup B \cup  X\cup Y\cup Z\cup T$, by  \cref{Vpart}. Moreover, the subsets of $V(G)$  defined above have several interesting and useful structural properties, and they are given in \cref{ato} to \cref{thm:C5-F123-free}  below.

	 	\begin{lemma}\label{ato}
For each $i\in \{1,2,3,4,5\}$, $i\mod 5$, the following hold:
\begin{enumerate}[label=  ($\mathbb{O}$\arabic*), leftmargin=1.25cm] \itemsep=0pt
\item \label{XiYiancom}   $A_{i+1}\cup A_{i-1}\cup X_i$  is anticomplete to $A_i\cup B_i\cup B_{i-1}\cup Y_{i}$.
\item \label{Ai+2Yicom} $A_{i+2}\cup A_{i-2}$ is complete to $A_i\cup B_{i-1}\cup B_{i}\cup Y_i$.
\item \label{ABTst} $A_i\cup A_{i+1}\cup B_i\cup  T$ is a stable set.
\item  \label{BnonempAemp} At least one of $B_i$ and  $(A\sm A_{i-2})\cup B_{i+1}\cup B_{i-1}$ is   empty.

\end{enumerate}
 	\end{lemma}

	\begin{proofthm}
	\ref{XiYiancom}:~If there are adjacent vertices, say $a\in A_{i+1}\cup A_{i-1}\cup X_i$ and $a'\in A_i\cup B_i \cup B_{i-1}\cup Y_{i}$, then $\{v_{i+2},v_{i-2},a,a',v_i\}$ induces a $P_2+P_3$. So \ref{XiYiancom} holds. \hfill{$\sq$}

\smallskip
\noindent{\ref{Ai+2Yicom}:}~If there are nonadjacent vertices, say $p\in A_{i+2}\cup A_{i-2}$ and $q\in A_i\cup B_{i-1}\cup B_i\cup Y_i$, then $\{q,v_i,p,v_{i+2},v_{i-2}\}$ induces a $P_2+P_3$. So \ref{Ai+2Yicom} holds. \hfill{$\sq$}

\smallskip
\noindent{\ref{ABTst}}:~If there are adjacent vertices in $A_i\cup A_{i+1}\cup B_i\cup T$, say $u$ and $v$, then $\{u,v,v_{i+2},v_{i-2},v_{i-1}\}$ induces a $P_2+P_3$. So \ref{ABTst} holds. \hfill{$\sq$}

\smallskip
\noindent{\ref{BnonempAemp}}:~Suppose not. Then we may assume  (up to symmetry) that there  are vertices, say $b\in B_i$ and $a\in A_i \cup A_{i+2}\cup B_{i+1}$. Now if $a\in A_i$, then $\{v_{i+2},v_{i-2},a,v_i,b\}$ induces a $P_2+P_3$ (by \ref{ABTst}). Next if $a\in A_{i+2}$, then $\{a,v_{i+2},b,v_i,v_{i-1}\}$ or $\{v_{i-1},v_{i-2},a,b,v_{i+1}\}$ induces a $P_2+P_3$.
Finally if $a\in B_{i+1}$, then  $\{v_{i-2},v_{i-1},a,v_{i+1},b\}$ induces a $P_2+P_3$ or $\{v_{i},b,a,v_{i+2},v_{i+1}\}$ induces a gem.
  These contradictions show that   \ref{BnonempAemp} holds.
	\end{proofthm}

	 	\begin{lemma}\label{lem:YZ}
For each $i\in \{1,2,3,4,5\}$, $i\mod 5$, the following hold:
\begin{enumerate}[label=  ($\mathbb{M}$\arabic*), leftmargin=1.25cm] \itemsep=0pt
\item \label{YiBi+2emp} At least one of $Y_i$ and $B_{i+1}\cup B_{i-2} \cup X_{i+2}\cup X_{i-2}$ is  empty.
\item \label{ZiYiBi+1oneemp}  At least one of $Z_i$ and $B_{i+1}\cup B_{i-2} \cup Y_{i+1} \cup Y_{i-1}$ is  empty.
\item 	\label{Yiclq}
		$Y_i$ is a clique, and $Y_i$ is complete to $Y_{i+1}\cup A_i\cup B_{i}\cup B_{i-1}\cup Z_{i+2}\cup Z_{i-2}$.
\item \label{YiYi+2ancom} $Y_i$ is anticomplete to $Y_{i+2}\cup Y_{i-2}\cup T$.
\item\label{YiYi+1-Ai-2} If $Y_i$ and $Y_{i+1}$ are nonempty, then $A_{i-2}$ is empty.
\item \label{ZiZi+1ancom} $Z_i$ is anticomplete to $(Z\sm Z_{i})\cup X_{i+1}\cup X_{i-1}$. Further, if $Z_i$ and $Z_{i+1}$ are nonempty, then $Z_i\cup Z_{i+1}$ is a stable set. Likewise, if $Z_i$ and $Z_{i-1}$ are nonempty, then $Z_i\cup Z_{i-1}$ is a stable set.
\item 		\label{YiBigcompZcom} The vertex-set of any big component  of  $G[Z_i]$ is complete to $Y_i$. Likewise, the vertex-set of any big component of $X_{i+1}$ is complete to $Y_i$, and the vertex-set of any big component of $X_{i-1}$ is complete to $Y_i$.

\item \label{Xi+2BigcompZcom} The vertex-set of any big component  of $G[X_i]$ is complete to $Z_{i+2}\cup Z_{i-2}$.
\end{enumerate}
\end{lemma}
\begin{proofthm}\ref{YiBi+2emp}:~Suppose not. Then we may assume  (up to symmetry) that there  are vertices, say  $y\in Y_i$ and $u\in B_{i+1} \cup X_{i+2}$. But then   $\{y,v_i,u,v_{i+2},v_{i-2}\}$ induces a $P_2+P_3$ or $\{v_{i-1},v_i, v_{i+1}, u, y\}$ induces a gem which is a contradiction. So  \ref{YiBi+2emp} holds. \hfill{$\sq$}

\smallskip
\noindent{\ref{ZiYiBi+1oneemp}}:~Suppose not. Then we may assume  (up to symmetry) that there  are vertices, say $z\in Z_i$ and $u\in B_{i+1}\cup Y_{i+1}$. But then $\{u,v_{i+1},z,v_{i-2},$ $v_{i-1}\}$ induces a $P_2+P_3$ or $\{v_{i+1},u,z,v_{i-2},v_{i+2}\}$ induces a gem which is a contradiction. So \ref{ZiYiBi+1oneemp} holds. \hfill{$\sq$}

\smallskip
\noindent{\ref{Yiclq}}:~Suppose not. Then  we may assume  (up to symmetry) that there are nonadjacent vertices, say $u\in Y_i$ and $w
\in Y_i\cup Y_{i+1}\cup A_i\cup B_{i}\cup Z_{i-2}$. Now if $w\in Y_i\cup A_i\cup B_i$, then $\{v_{i+2},v_{i-2},u,v_i,w\}$ induces a $P_2+P_3$, and if  $w\in Y_{i+1}\cup Z_{i-2}$, then $\{v_{i-1},u,v_{i+1},w,v_{i}\}$ induces a gem.   These contradictions show that \ref{Yiclq} holds. \hfill{$\sq$}

\smallskip
\noindent{\ref{YiYi+2ancom}}:~Suppose not. Then  we may assume (up to symmetry) that there are adjacent vertices, say $u\in Y_i$ and $w
\in  Y_{i+2}\cup T$.  Now if $w\in Y_{i+2}$, then $\{v_i,u,w,v_{i+2},v_{i+1}\}$ induces a gem, and if $w\in T$,  then $\{v_{i+2},v_{i-2}, w,u,v_{i}\}$ induces a $P_2+P_3$. These contradictions show that \ref{YiYi+2ancom} holds. \hfill{$\sq$}

\smallskip
\noindent{\ref{YiYi+1-Ai-2}}:~Suppose not, and let $a\in A_{i-2}$. Then since $Y_i$ and $Y_{i+1}$ are nonempty, there are vertices, say $y\in Y_i$   and $y'\in Y_{i+1}$, and then $\{v_{i-1},v_i,y',a, y\}$ induces a gem (by  \ref{Ai+2Yicom} and \ref{Yiclq}) which is a contradiction. So \ref{YiYi+1-Ai-2} holds. \hfill{$\sq$}

 \smallskip
\noindent{\ref{ZiZi+1ancom}}:~If there are adjacent vertices, say $z\in Z_i$ and $u\in Z_{i+1}\cup  Z_{i+2} \cup X_{i+1}$, then $\{v_{i-1},u,z,v_{i+2},v_{i-2}\}$ or $\{v_i,u,v_{i+2},v_{i-2},z\}$ induces a gem; so $Z_i$ is anticomplete to $Z_{i+1}\cup  Z_{i+2} \cup X_{i+1}$. Likewise, $Z_i$ is  anticomplete to $Z_{i-1}\cup  Z_{i-2} \cup X_{i-1}$. This proves the first assertion of  \ref{ZiZi+1ancom}. 	
	
		Next suppose to the contrary that   there are adjacent vertices in $Z_i\cup Z_{i+1}$, say $u$ and $v$. Then by the first assertion  of  \ref{ZiZi+1ancom}, we may assume that $u,v\in Z_i$, and then for any $w\in Z_{i+1}$, we see that   $\{u,v,v_{i+1},w,v_{i-1}\}$ induces a $P_2+P_3$ (by the first assertion  of  \ref{ZiZi+1ancom}) which is a contradiction.  So $Z_i\cup Z_{i+1}$ is a stable set. This proves the second assertion of  \ref{ZiZi+1ancom}. \hfill{$\sq$}

 \smallskip
\noindent{\ref{YiBigcompZcom}}:~Let $Q$ be a big component of $G[Z_i]$. Suppose to the contrary that $V(Q)$ is not complete to $Y_i$. Then there are vertices, say $p,q\in V(Q)$ and $r\in Y_i$ such that $pq\in E(G)$ and $pr\notin E(G)$. Then since $\{p,q,v_{i-1},r,v_{i+1}\}$ does not induce a $P_2+P_3$, we have $qr\in E(G)$, and then $\{p,q,r,v_{i+1},v_i\}$ induces a gem which is a contradiction. So \ref{YiBigcompZcom} holds. \hfill{$\sq$}

\smallskip
\noindent{\ref{Xi+2BigcompZcom}}:~Let $Q$ be a big component of $G[X_i]$. Suppose to the contrary that $V(Q)$ is not complete to $Z_{i+2}$. Then there are vertices, say $p,q\in V(Q)$ and $r\in Z_{i+2}$ such that $pq\in E(G)$ and $pr\notin E(G)$. Then since $\{p,q,v_{i-2},v_{i+2},r\}$ does not induce a $P_2+P_3$, we have $qr\in E(G)$, and then $\{r,v_{i-1},p,v_{i+1},q\}$ induces a gem which is a contradiction. Thus $V(Q)$ is complete to $Z_{i+2}$. Likewise, $V(Q)$ is complete to $Z_{i-2}$. This proves \ref{Xi+2BigcompZcom}.
\end{proofthm}

	\begin{lemma}\label{lem:per}
For each $i\in \{1,2,3,4,5\}$, $i\mod 5$, the following hold:
\begin{enumerate}[label= ($\mathbb{L}$\arabic*), leftmargin=1.25cm] \itemsep=0pt
\item 	\label{XiP3free} $G[X_i]$ is   $P_3$-free,  and hence it is the union of vertex-disjoint complete graphs.
\item \label{XiAiAi2col}
		If $X_i$ is a stable set, then   $\chi(G[A_i\cup A_{i+1}\cup A_{i-1}\cup B_{i}\cup B_{i-1}\cup X_i\cup \{v_i,v_{i+2},v_{i-2}\}])\leq 2$.
\item \label{Ziper}	$G[Z_i]$ is  perfect, and hence $\chi(G[Z_i])=\omega(G[Z_i])\leq \omega(G)-2$.
\end{enumerate}
	\end{lemma}
	\begin{proofthm}
		\ref{XiP3free}:~If there are vertices in $X_i$, say $p,q$ and $r$ such that $\{p,q,r\}$ induces a $P_3$, then $\{v_{i+2},v_{i-2},p,$ $q,r\}$ induces a $P_2+P_3$; so $G[X_i]$ is $P_3$-free. This proves \ref{XiP3free}. \hfill{$\sq$}

\smallskip
\noindent{\ref{XiAiAi2col}}:~We will show for $i=1$. Clearly from \ref{ABTst}, $\chi(G[A_1\cup B_1\cup B_5\cup \{v_3,v_4\}])\leq
\chi(G[A_1\cup B_1\cup \{v_4\}])+\chi(G[B_5\cup \{v_3\}]) \leq 2$. Next we will show that
 $\chi(G[A_2\cup A_5\cup  X_1\cup \{v_1\}])\leq 2$. First let $X_1':=\{x\in X_1\mid N(x)\cap A_2\neq \es\}$. Now if there are adjacent vertices, say $u\in A_5$ and $w\in X_1'$, then for any $r\in N(w)\cap A_2$, we see that $\{v_2,r,u,v_5,w\}$ induces a gem (by \ref{Ai+2Yicom}); so $A_5$ is anticomplete to $X_1'$. Thus $A_5\cup X_1'$ and $A_2\cup (X_1\sm X_1')$ are stable sets (by \ref{ABTst}), and so we have $\chi(G[A_2\cup A_5\cup X_1\cup \{v_1\}])\leq 2$. Then since $A_1\cup B_1\cup B_5\cup \{v_3,v_4\}$ is anticomplete to $A_2\cup A_5\cup X_1\cup \{v_1\}$ (by \ref{XiYiancom}), we conclude that $\chi(G[A_1\cup B_{1}\cup B_{5}\cup A_{2}\cup A_{5}  \cup X_1\cup  \{v_1,v_{3},v_{4}\}])\leq 2$.
 This proves \ref{XiAiAi2col}. \hfill{$\sq$}

\smallskip
\noindent{\ref{Ziper}}:~This follows since $Z_i$ is complete to $\{v_{i-2}, v_{i+2}\}$, $G[Z_i]$ is   $P_4$-free, and hence perfect.
	\end{proofthm}

\begin{lemma}\label{W45}
		Let $W:=\{j\in \{1,2,3,4,5\}\mid Y_j\neq \es\}$. Then the following hold:
\begin{enumerate}[label= (\roman*), leftmargin=0.85cm] \itemsep=0pt
\item\label{W45-5} If  $|W|=5$, then  $G\in \cal C$, and hence $\chi(G)\leq \lceil\frac{5\omega(G)-1}{4}\rceil$.

\item\label{W45-4} If  $|W|=4$, then  $\chi(G)\leq \omega(G)+1$.
\end{enumerate}
		\end{lemma}
		\begin{proofthm}For $i\in \{1,2,3,4,5\}$, $i$ mod 5, we let $R_i:=Y_i\cup \{v_i\}$.\\ $\ref{W45-5}$:~Suppose that $|W|=5$. Then from \ref{YiBi+2emp},  \ref{ZiYiBi+1oneemp} and \ref{YiYi+1-Ai-2}, we have $A\cup B\cup X\cup Z=\es$. Also from \ref{Yiclq} and \ref{YiYi+2ancom}, we see that $G[\cup_{i=1}^5 R_i]$ is a  clique blowup of a $C_5$, and  thus $G[(\cup_{i=1}^5R_i) \cup T]$ is isomorphic to $G^*+\ell K_1$, where $G^*$ is a clique blowup of a $C_5$. So $G\in \cal C$ and hence  $\chi(G)\leq \lceil\frac{5\omega(G)-1}{4}\rceil$. This proves $\ref{W45-5}$. \hfill{$\sq$}

	\smallskip
\noindent{$\ref{W45-4}$}:~Suppose that $|W|=4$, and we may assume that $Y_5=\es$. Then  from \ref{YiBi+2emp},  \ref{ZiYiBi+1oneemp} and \ref{YiYi+1-Ai-2}, we have $A_1\cup A_4\cup A_5\cup B\cup X\cup Z=\es$. Also from \ref{Yiclq} and \ref{YiYi+2ancom}, we see that $G[\cup_{i=1}^4R_i]$ is a clique blowup of a $P_4$, and  hence $G[\cup_{i=1}^4R_i]$ is perfect.  Hence $\chi(G)\leq \chi(G[\cup_{i=1}^4R_i])+\chi(G[A_2\cup A_3\cup T\cup \{v_5\}])\leq \omega(G)+1$, by \ref{ABTst}. This proves $\ref{W45-4}$.
\end{proofthm}

\begin{lemma}\label{F3claim}
If further $G$ is $F_1$-free, then  for each $i\in \{1,2,3,4,5\}$, $i\mod 5$, the following hold:
\begin{enumerate}[label=(\roman*), leftmargin=0.85cm]\itemsep=0pt
    \item\label{F3fr-Yemp} $Y$ is an empty set.
  \item\label{F3fr-Xi}  $X_i$ is a stable set.
   \item\label{F3fr-XB}  At least one of $X_i$ and $B_{i-2}\cup B_{i+1}$ is  empty.
\end{enumerate}
\end{lemma}
\begin{proofthm}$\ref{F3fr-Yemp}$:~If there is a vertex, say $y\in Y$, then we may assume that $y\in Y_i$, and then $C\cup \{y\}$ induces an $F_1$; so $Y$ is an empty set. This proves $\ref{F3fr-Yemp}$. \hfill{$\sq$}

\smallskip
\noindent{$\ref{F3fr-Xi}$}:~If there are  adjacent vertices in $X_i$, say $x$ and $x'$, then $(C\sm \{v_i\})\cup \{x,x'\}$  induces an $F_1$; so  $X_i$ is a stable set. This proves $\ref{F3fr-Xi}$.  \hfill{$\sq$}

\smallskip
\noindent{$\ref{F3fr-XB}$}:~If there are  vertices, say $x\in X_i$ and $b\in B_{i-2}\cup B_{i+1}$, then since $\{v_{i+2}, b, x,v_{i-1},v_i\}$ or  $\{v_{i-2},b,x,v_{i+1},$ $v_i\}$ does not induce a $P_2+P_3$, we have $xb\in E(G)$, and then $(C\sm \{v_i\})\cup \{x,b\}$ induces an $F_1$; so at least one of $X_i$ and $B_{i-2}\cup B_{i+1}$ is  empty.
 \end{proofthm}

 \begin{lemma}\label{thm:C5-F123-free}
	If further $G$ is  $(F_1,F_2,F_3)$-free, then $\omega(G)\leq 3$.
	\end{lemma}
	\begin{proofthm}
		 First since $G$ is $F_1$-free, from \cref{F3claim}, we have $Y=\es$  and $X_i$ is a stable set for   $i\in \{1,2,3,4,5\}$. Also  for any two adjacent vertices, say $z,z'\in Z_i$, since $C\cup \{z,z'\}$ does not induce an $F_2$, we see that  $Z_i$ is a stable set for    $i\in \{1,2,3,4,5\}$. Now suppose to the contrary that there is a clique of size at least $4$ in $G$, say $Q$.   Then we claim the following:

		\begin{claim}\label{C5-F123-ZQ}
			$Q\cap Z$ is an empty set.
		\end{claim}
		\begin{proof}Suppose to the contrary there is a vertex, say  $z\in Q\cap Z$. We may assume that $z\in Z_1$. Then we have the following observations:
 \begin{enumerate}[label=($\roman*$)]\itemsep=0pt 
 \item\label{C5-F123-ZQ1} From \ref{ZiYiBi+1oneemp}  and \ref{ZiZi+1ancom}  and since $Z_1$ is a stable set, we have $Q\cap ((B_2\cup B_4)\cup (Z\sm \{z\})\cup X_2\cup X_5)=\es$. Also for any $r\in Q\cap B_3$,  since $C\cup \{z,r\}$ does not induce an $F_3$, we have $Q\cap B_3=\es$.   These arguments also imply that $Q\sm \{z\}$ is not complete to $\{v_3,v_4\}$.
\item\label{C5-F123-ZQ2} If $(Q\sm \{z\})\cap (A_3\cup A_4\cup X_3\cup X_4\cup T)=\es$, then $Q\subseteq S$, where $S:=C\cup A_1\cup A_2\cup A_5\cup B_1\cup B_5\cup X_1\cup \{z\}$, and then since
$\omega(G[S])\leq \chi(G[S])\leq \chi(G[A_1\cup A_2\cup A_5\cup B_1\cup B_5\cup X_1\cup \{v_1,v_3,v_4\}])+\chi(G[\{z,v_2,v_5\}])= 2+1=3$ (by \ref{XiAiAi2col}), we see that $|Q|\leq 3$ which is a contradiction; so $(Q\sm \{z\})\cap (A_3\cup A_4\cup X_3\cup X_4\cup T)\neq \es$.
\item\label{C5-F123-ZQ3}  If $(Q\sm \{z\})\cap (A_3\cup A_4\cup X_3\cup X_4)=\es$, then from  \ref{C5-F123-ZQ1} and \ref{C5-F123-ZQ2}, $(Q\sm \{z\})\cap T\neq \es$, and then $Q\subseteq \{z\}\cup X_1\cup T$ (by \ref{ABTst}) and thus since $\omega(G[\{z\}\cup X_1\cup  T])\leq 3$ (by \ref{ABTst}), we have $|Q|\leq 3$ which is  a contradiction; so $(Q\sm \{z\})\cap (A_3\cup A_4\cup  X_3\cup X_4)\neq \es$.  
\end{enumerate}
Now  we prove that $Q\sm \{z\}$ is anticomplete to $\{v_3,v_4\}$. Suppose to the contrary that there is a vertex in $Q\sm \{z\}$, say $a$, such that $\{a\}$ is not anticomplete to $\{v_3,v_4\}$, and  we may assume that $av_3\in E(G)$ and $av_4\notin E(G)$; so $a\in A_3\cup X_4$. Then for any $b\in Q\sm \{a,z\}$, since $\{b,a,v_3,v_4,z\}$ does not induce a gem, we see that every vertex in  $Q\sm \{a,z\}$ is adjacent to exactly one of $v_3$ and $v_4$. Hence $Q\sm \{z\} \subseteq A_3\cup A_4\cup  X_3\cup X_4\cup \{v_3,v_4\}$. Now if there are adjacent vertices, say $u\in (Q\sm \{z\})\cap A_3$ and $w\in (Q\sm \{z\})\cap X_4$, then $\{w,v_5,v_1,v_2,v_3,z,u\}$ induces an $F_3$; so $(Q\sm \{z\})\cap A_3$ is anticomplete to $(Q\sm \{z\})\cap X_4$.  Likewise, $(Q\sm \{z\})\cap A_4$ is anticomplete to $(Q\sm \{z\})\cap X_3$.  Thus from \ref{ABTst}, we observe that $(Q\sm \{z\})\cap ( A_3\cup A_4\cup  X_3\cup X_4\cup \{v_3,v_4\})$ induces a bipartite graph, and hence $|Q|\leq 3$ which is a contradiction.  So $Q\sm \{z\}$ is anticomplete to $\{v_3,v_4\}$. This is a contradiction to the fact that $Q\cap (A_3\cup A_4\cup  X_3\cup X_4)\neq \es$, and hence $Q\sm \{z\}$ is not anticomplete to $\{v_3,v_4\}$. This proves \cref{C5-F123-ZQ}.
\end{proof}

		\begin{claim}\label{C5-F123-XQ}
			For each $i\in \{1,2,3,4,5\}$ and $i$ mod $5$, at least one of  $Q\cap X_{i+1}$ and $Q\cap X_{i-1}$ is  empty.
		\end{claim}
		\begin{proof}We will show for $i=1$. Suppose to the contrary that there are vertices, say $a\in Q\cap X_2$ and $b\in Q\cap X_5$. Then from  \ref{XiYiancom}, \cref{F3claim}:$\ref{F3fr-XB}$ and \cref{C5-F123-ZQ}, $Q\cap (A_2\cup A_5\cup B\cup Z)=\es$. Also for any $c\in (Q\sm \{a,b\})\cap (A_1\cup A_4\cup X_3)$, since  $\{v_1,a,c,v_4,b\}$ does not  induce a gem, we have   $Q\cap(A_1\cup A_4\cup X_3)=\es$. Likewise, $Q\cap (A_3\cup X_4)=\es$. Thus since $|Q|\geq 4$, we have $Q\cap X_1\neq \es$ and $Q\cap T\neq \es$. Then for any $u\in Q\cap X_1$ and $w\in Q\cap T$, we observe that $\{v_4,v_5,u, a,v_3,b,w\}$ induces an $F_3$ which is a contradiction. This proves \cref{C5-F123-XQ}. \end{proof}

\begin{claim}\label{C5caseAiXicol}
			For each $i\in \{1,2,3,4,5\}$, we have $\chi(G[A\cup B\cup C\cup T\cup X_i])\leq 3$.
		\end{claim}
	\begin{proof}We will prove for $i=1$. Since $A_{3}\cup A_{4}\cup B_{3}\cup T\cup \{v_{2},v_{5}\}$ is a stable set (by \ref{ABTst}), it is enough to prove that $\chi(G[A_{1}\cup A_{2}\cup A_{5}\cup B_{1}\cup B_{2}\cup B_{4}\cup B_{5}\cup X_1\cup \{v_1,v_{3},v_{4}\}])\leq 2$. From \ref{XiAiAi2col}, we may assume that $B_{2}\cup B_{4}\neq\es$, and by using symmetry, we may assume that  $B_{2}\neq \es$. Then  $A_{1}\cup A_{2}\cup B_{1} \cup X_1=\es$ (by \ref{BnonempAemp} and \cref{F3claim}:$\ref{F3fr-XB}$), and then from \ref{ABTst} and \ref{BnonempAemp}, we see that $\chi(G[A_{1}\cup A_{2}\cup A_{5}\cup B_{1}\cup B_{2}\cup B_{4}\cup B_{5}\cup X_1\cup \{v_1,v_{3},v_{4}\}])\leq \chi(G[A_{5}\cup B_{4}\cup B_{5}\cup \{v_{3}\}])+ \chi(G[B_{2}\cup \{v_1,v_{4}\}])= 1+1=2$. This proves \cref{C5caseAiXicol}. \end{proof}

		\medskip
		Now from  \cref{C5-F123-ZQ} and \cref{C5caseAiXicol}, since $Q\not\subseteq A\cup B\cup C\cup T\cup X_i$, for exactly one $i\in \{1,2,3,4,5\}$, there are indices $j,\ell\in [5]$ with $j\neq \ell$  such that  both $Q\cap X_j$ and $Q\cap X_{\ell}$ are nonempty. Then from  \cref{C5-F123-XQ}, we may assume that $Q\cap X_1\neq \es$, $Q\cap X_2\neq \es$ and $Q\cap (X\sm (X_1\cup X_2))=\es$, and let $a\in Q\cap X_1$ and $b\in Q\cap X_2$.  Then clearly $Q\cap C=\es$, and  $B\sm B_1=\es$ (by \cref{F3claim}:$\ref{F3fr-XB}$).  Also from \ref{XiYiancom}, we have $Q\cap (A_1\cup A_2\cup B_1)=\es$. Moreover,  for any $u\in Q\cap A_3$, since $\{a,b,v_3,v_4,v_5,u\}$ does not induce an $F_1$, we have $Q\cap A_3=\es$. Likewise, $Q\cap A_5=\es$. So from \cref{C5-F123-ZQ} and from above arguments, we conclude that $Q\sm \{a,b\} \subseteq  A_4\cup T$, and hence $|Q\sm \{a,b\}|\leq 1$ (by \ref{ABTst}) which is a contradiction to our assumption that $|Q|\geq 4$. 		So we have $\omega(G)\leq 3$. This completes the proof of \cref{thm:C5-F123-free}.
	\end{proofthm}
	
	\section{Coloring $(P_2+P_3$, gem$)$-free graphs: Proof of \cref{structhm}}\label{sec:C5}

In this section, we give a proof of \cref{structhm} (and it is given at the end of this section). Note that if a $(P_2+P_3$, gem$)$-free graph $G$ contains a $C_5$, then from \cref{thm:C5-F123-free} and \cref{thm:p2p3K4-bnd}, we may assume that $G$ contains at least one of $F_1$, $F_2$ and $F_3$.

\setcounter{claim}{0}

	\subsection{$(P_2+P_3$, gem$)$-free graphs that contain an $F_1$}

We start with the following.

	\begin{lemma}\label{lem:F11}
		Let $G$ be a $(P_2+P_3$, gem$)$-free graph. If $G$ contains an $F_1'$, then either $G\in \cal C$  or $\chi(G)\leq \omega(G)+1$.
	\end{lemma}
	\begin{proofthm}
		Let $G$ be a $(P_2+P_3$, gem$)$-free graph. Suppose that $G$ contains an $F_1'$ with vertices and edges as shown in Figure~\ref{fig-F123}. Let $C:=\{v_1,v_2,v_3,v_4,v_5\}$. Then, with respect to $C$, we define the sets $A$, $B$, $X$, $Y$, $Z$ and $T$ as in Section~\ref{genprop}, and we use the properties in Section~\ref{genprop}.  Clearly $y_1^*\in Y_1$ and $y_2^*\in Y_2$   so that $Y_1$ and $Y_2$ are nonempty.  Then $(B\sm B_1)\cup X_3\cup X_4\cup X_5=\es$ (by \ref{YiBi+2emp}),  $Z_1\cup Z_2\cup Z_3\cup Z_5=\es$ (by \ref{ZiYiBi+1oneemp}), and $A_4=\es$ (by  \ref{YiYi+1-Ai-2}).
		For each $i$, we let $R_i:=Y_i\cup \{v_i\}$, and let $W:=\{j\in \{1,2,3,4,5\}\mid Y_j\neq \es\}$. Now if $|W|\in \{4,5\}$, then by using \cref{W45}, we conclude the proof. So we may assume that   $|W|\in \{2,3\}$. Thus at least one of $Y_3$ and $Y_5$ is empty, and we may assume that $Y_5=\es$.
		For $j\in \{1,2\}$, we let $X_j'$ be the set $\{u\in X_j\mid N(u)\cap (X_j\sm \{u\})\neq \es\}$. Then $X_j\sm X_j'$ is a stable set which is anticomplete to $X_j'$ (by \ref{XiP3free}), and recall that $X_j'$ is complete to $R_{j-1}\cup R_{j+1}$, and is anticomplete to $R_j\cup R_{j+2}$, by \ref{XiYiancom}, \ref{YiBi+2emp}, \ref{Yiclq} and \ref{YiBigcompZcom}. Hence $R_1$,  $R_3$ and $X_1'$ are mutually anticomplete to each other (by \ref{YiYi+2ancom}). Also from \ref{Yiclq}, $R_i$'s are cliques, and $R_2$ is complete to $R_1\cup R_3$. To proceed further, we let:
\begin{center}
 \begin{tabular}{l}
  $S:=$ $(X_1'\cup R_1\cup R_3)\cup (X_2'\cup R_2),$\\
 $ q_1 :=$  $\max\{\omega(G[R_1]), \omega(G[R_3]), \omega(G[X_1'])\},$\\

    $q_2:=$  $\max\{\omega(G[R_2]), \omega(G[R_4])-q_1, \omega(G[X_2'])\}$, \\
      $q_2':=$  $\max\{\omega(G[R_2]),  \omega(G[X_2'])\}$,  and  \\
 $ q_3 :=$  $\omega(G[Z_4])$.
 \end{tabular}
 \end{center}
Note that $q_1\geq 2$, $q_2\geq 2$ and $q_2'\geq 2$. Also   $\omega(G[R_1\cup R_3\cup X_1'])=q_1$,   $\omega(G[R_2\cup R_4\cup X_2'])\leq q_1+q_2$ (if $Y_4\neq \es$), and $\omega(G[R_2\cup R_4\cup X_2'])= q_2'$ (if $Y_4= \es$). From \ref{Ziper}, $\chi(G[Z_4])=\omega(G[Z_4])=q_3$. Since $A_1\cup A_2\cup B_1\cup T\cup \{v_3, v_5\}$ is a stable set (by \ref{ABTst}), we have $\chi(G[A_1\cup A_2\cup B_1\cup T\cup \{v_3, v_5\}]) = 1$.
				Next we claim the following:
\begin{claim}\label{S-clq}
If there are cliques, say $K$ in $G[X_1']$ and  $K'$ in $G[X_2']$ with $|K|\geq 3$ and $|K'|\geq 3$, then $K$ is complete to $K'$. Hence, $S$ contains a  clique of size $q_1+q_2'$.
\end{claim}
\begin{proof}
		Suppose to the contrary that there are nonadjacent vertices, say $a\in K$ and $a'\in K'$. Since $|K|\geq 3$, there   is a vertex in $K\sm \{a\}$, say $b$. Also   for any $u\in K$ and $v,w\in K'$, since $\{v,w,v_2,u,v_5\}$ does not induce a $P_2+P_3$, we see that each vertex of $K$ has at most one nonneighbor in $K'$. Likewise, each vertex of $K'$ has at most one nonneighbor in $K$. Thus $a'b\in E(G)$ and since $|K'|\geq 3$, $a$ and $b$ has a common neighbor in $K'$, say $b'$. Then $\{a,b,a',v_3,b'\}$ induces a gem which is a contradiction. So $K$ is complete to $K'$. This proves the first assertion of \cref{S-clq}.

Now if $\omega(G[X_1'])\leq 2$, since $q_1\geq 2$, $(R_1\cup R_3) \cup (R_2\cup X_2')$ contains a clique of size $q_1+q_2'$. Likewise, if $\omega(G[X_2'])\leq 2$, since $q_2'\geq 2$, $(X_1'\cup R_1\cup R_3) \cup R_2$ contains a clique of size $q_1+q_2'$.  So we may assume that $\omega(G[X_1'])\geq 3$ and $\omega(G[X_2'])\geq 3$. Then from the first assertion, we conclude that $S$ contains a clique of size $q_1+q_2'$. This proves the second assertion of \cref{S-clq}.
\end{proof}

		\begin{claim}\label{F4caseclq}
			$G$ contains a clique of size $q_1+q_2+q_3$.
		\end{claim}
		\begin{proof}First suppose that $q_2=\omega(G[R_4])-q_1$; so $\omega(G[R_4])=q_1+q_2$. Now if $q_3\leq 1$, then $R_4\cup \{v_3\}$ contains a clique of size $q_1+q_2+q_3$, and if $q_3\geq 2$, then $R_4\cup Z_4$ contains a  clique of size $q_1+q_2+q_3$, by  \ref{YiBigcompZcom}, and we are done.
		So we may assume that $q_2=\max\{\omega(G[R_2]),\omega(G[X_2'])\}=q_2'$. Then from \cref{S-clq}, $S$ contains a clique of size $q_1+q_2$.  Thus if $Z_4=\es$, then $q_3=0$, and we are done; so we may assume that $Z_4\neq \es$. Then from   \ref{ZiYiBi+1oneemp}, we have $R_3=\{v_3\}$, and from  \ref{Yiclq} and \ref{Xi+2BigcompZcom},   $Z_4$ is complete to $S\sm R_3$.  Then since $q_1\geq 2$, $(X_1'\cup R_1)\cup (X_2'\cup R_2)\cup Z_4$ contains a clique of size $q_1+q_2+q_3$, and we are done. This proves \cref{F4caseclq}.
\end{proof}

		\begin{claim}\label{F4casecolst}
			$\chi(G[A_5\cup X_1\cup R_1\cup R_3])\leq q_1$. Likewise, 	$\chi(G[A_3\cup X_2\cup R_2\cup \{v_4\}])\leq q_2$.
		\end{claim}		
	\begin{proof}  Since  $G[R_1\cup R_3]$ is a clique blowup of a $2K_1$ (by \ref{Yiclq} and \ref{YiYi+2ancom}), it is a perfect graph, and hence $\chi(G[R_1\cup R_3])= \omega(G[R_1\cup R_3]) \leq q_1$. Also from \ref{XiYiancom}, \ref{YiBi+2emp} and \ref{XiP3free}, we see that $\chi(G[X_1'\cup R_1\cup R_3]) \leq \max\{\chi(G[X_1']), \chi(G[R_1\cup R_3])\}\leq q_1$, and $X_1\sm X_1'$ is anticomplete to $X_1'\cup R_1\cup R_3$. Moreover since $X_1\sm X_1'$ is a stable set, we have $\chi(G[A_5\cup (X_1\sm X_1')])\leq 2$ (by \ref{ABTst}). Next if there are adjacent vertices, say $a\in A_5$ and $x\in X_1'$, then from \ref{Ai+2Yicom} and \ref{YiBigcompZcom},  $\{v_5,a,y_2^*,v_2,x\}$ induces a gem; so $A_5$ is anticomplete to $X_1'$. Now since $q_1\geq 2$ and since $A_5$ is anticomplete to $X_1'\cup R_1\cup R_3$ (by \ref{XiYiancom} and \ref{YiYi+1-Ai-2}), we conclude that $\chi(G[A_5\cup X_1\cup R_1\cup R_3])\leq \max\{\chi(G[X_1'\cup R_1\cup R_3]),\chi(G[A_5\cup (X_1\sm X_1')])\}\leq \max\{q_1,2\}\leq q_1$.  This proves \cref{F4casecolst}. \end{proof}

\begin{claim}\label{F1-Y4emp}
We may assume that $Y_4$ is an empty set.
\end{claim}
\begin{proof}
Suppose not, and let $u\in Y_4$. Then we have the following:
 \begin{enumerate}[label=($\roman*$)]\itemsep=0pt
 \item\label{F1Y4-X12Y3} $X_1\cup X_2=\es$ (by \ref{YiBi+2emp}), and since $|W|\leq 3$, we have $Y_3=\es$.
  \item\label{F1Y4-A35} If there is a vertex in $A_3$, say $w$, then $wy_1^*\in E(G)$ (by \ref{Ai+2Yicom}) and $wu\notin E(G)$ (by \ref{XiYiancom}), and then $\{u,v_4,w,y_1^*,v_1\}$ induces a $P_2+P_3$; so $A_3=\es$. Likewise, by using $y_2^*$, we see that $A_5=\es$.

\item\label{F1Y4-R124} Since  $G[R_1\cup R_2\cup R_4]$ is a clique blowup of a $K_2+K_1$ (by \ref{Yiclq} and \ref{YiYi+2ancom}), it is a perfect graph, and hence $\chi(G[R_1\cup R_2\cup R_4]) = \max\{\omega(G[R_1\cup R_2]), \omega(G[R_4])\} \leq q_1+q_2$.
        \end{enumerate}
  Now clearly  $\chi(G)\leq \chi(G[R_1\cup R_2\cup R_4])+\chi(G[Z_4])+\chi(G[A_1\cup A_2\cup B_1\cup T\cup \{v_3,v_5\}])$.  Thus from \ref{F1Y4-X12Y3}, \ref{F1Y4-A35} and \ref{F1Y4-R124}, and from   \cref{F4caseclq}, we conclude that $\chi(G) \leq q_1+q_2+q_3+1\leq \omega(G)+1$, and we are done.  So we may assume that $Y_4$ is an empty set.
\end{proof}
		
		\medskip
		Now by using \cref{F1-Y4emp}, clearly $\chi(G)\leq \chi(G[A_5\cup X_1\cup R_1\cup R_3])+\chi(G[A_3\cup X_2\cup R_2\cup \{v_4\}])+\chi(G[Z_4])+\chi(G[A_1\cup A_2\cup B_1\cup T\cup \{v_5\}])$. Thus from \ref{ABTst}, \cref{F4caseclq} and \cref{F4casecolst}, we conclude that $\chi(G)\leq   q_1+q_2+q_3+1\leq \omega(G)+1$.
		This completes the proof of \cref{lem:F11}.
	\end{proofthm}

\setcounter{claim}{0}

	\begin{lemma}\label{lem:F12}
		Let $G$ be a ($P_2+P_3$, gem)-free graph. If $G$ contains an $F_1''$,  then either $G\in \cal C$  or $\chi(G)\leq \omega(G)+1$.
	\end{lemma}
	\begin{proofthm}
		Let $G$ be a ($P_2+P_3$, gem)-free graph. From \cref{lem:F11}, we may assume that $G$ is $F_1'$-free. Suppose that $G$ contains an $F_1''$ with vertices and edges as shown in Figure~\ref{fig-F123}. Let $C:=\{v_1,v_2,v_3,v_4,v_5\}$. Then, with respect to $C$, we define the sets $A$, $B$, $X$, $Y$, $Z$ and $T$ as in Section~\ref{genprop}, and we use the properties in Section~\ref{genprop}.  Clearly $y_1^*\in Y_1$ and $y_3^*\in Y_3$   so that $Y_1$ and $Y_3$ are nonempty. Then $B_1\cup B_2\cup B_4\cup X_1\cup X_3\cup X_4\cup X_5=\es$ (by
		\ref{YiBi+2emp}) and $Z_2\cup Z_4\cup Z_5=\es$ (by \ref{ZiYiBi+1oneemp}). Moreover  we have the following claims.

\begin{claim}\label{F12-Y245}
$Y_2\cup Y_4\cup Y_5$ is an empty set.
\end{claim}
\begin{proof}
If there is a vertex, say $u\in Y_2\cup Y_4\cup Y_5$, then from \ref{Yiclq}, we see that $C\cup \{y_1^*, y_3^*, u\} $ induces an $F_1'$; so $Y_2\cup Y_4\cup Y_5=\es$. This proves \cref{F12-Y245}.
\end{proof}

		\begin{claim}\label{F2-A45emp}
			$A_4\cup A_5$ is an empty set.
		\end{claim}
	\begin{proof}If there is a vertex, say $a\in A_4$, then from \ref{XiYiancom} and \ref{Ai+2Yicom}, $\{y_3^*,v_3,a,y_1^*,v_1\}$  induces a $P_2+P_3$; so  $A_4=\es$. Likewise, $A_5=\es$. This proves \cref{F2-A45emp}. \end{proof}
		
		\begin{claim}\label{F2-X2}
			$X_2$ is a stable set.
		\end{claim}
		\begin{proof}If there are adjacent vertices in $X_2$, say $x$ and $x'$, then from \ref{YiBigcompZcom}, we see that $\{v_1,x,v_3,v_4, $ $v_5,y_1^*,x'\}$ induces an $F_1'$; so $X_2$ is a stable set. This proves \cref{F2-X2}. \end{proof}
	
\medskip	
 To proceed further, we define the following: For $j\in \{1,3\}$, we let $q_j:=\omega(G[Y_i])$ and  $q_j':=\omega(G[Z_i])$. Also we let $q:=\max\{q_1,q_3\}$ and $q':=\max\{q_1',q_3'\}$.
		Then:
		
		\begin{claim}\label{F2-q13-q'13}
			We may assume that either $q_1\leq q_3$ or $q_1'\leq q_3'$. Likewise, we may assume that either $q_3\leq q_1$ or $q_3'\leq q_1'$.
		\end{claim}	
	\begin{proof}Suppose that $q_1>q_3$ and $q_1'>q_3'$. We will show that $\chi(G)\leq \omega(G)+1$, and we first observe the following:
\begin{enumerate}[label=($\roman*$)]\itemsep=0pt
 \item\label{q-col1}Recall that $A_2$ is stable set which is anticomplete to $Y_1\cup Y_3\cup \{v_4\}$, by \ref{XiYiancom}.  So from \ref{Yiclq} and \ref{YiYi+2ancom}, we have $\chi(G[A_2\cup Y_1\cup Y_3\cup \{v_4\}])= \max\{\chi(G[A_2]), \chi(G[Y_1]),$ $ \chi(G[Y_3\cup \{v_4\}])\}= \max\{1, q_1, q_3+1\} \leq q_1$.
     \item\label{q-col2} From \ref{Ziper}, \ref{ZiZi+1ancom} and from \cref{F2-X2}, we have $\chi(G[X_2\cup Z_1\cup Z_3\cup \{v_2,v_5\}]) = \max\{\chi(G[X_2\cup \{v_2\}]), \chi(G[Z_1]), \chi(G[Z_3\cup \{v_5\}])\} =\max\{1, \omega(G[Z_1]), \omega(G[Z_3\cup \{v_5\}])\} = \max\{1,q_1', q_3'+1\}\leq q_1'$.
     \item\label{q-col3} From \ref{q-col1}, \ref{q-col2}, \ref{ABTst}, \ref{BnonempAemp} and \cref{F2-A45emp}, we have $\chi(G)\leq \chi(G[A_2\cup Y_1\cup Y_3\cup \{v_4\}])+\chi(G[X_2\cup Z_1\cup Z_3\cup \{v_2,v_5\}])+\chi(G[A_1\cup B_5\cup T\cup \{v_3\}])+\chi(G[A_3\cup B_3\cup \{v_1\}])\leq q_1+q_1'+1+1 =q_1+q_1'+2$.
 \end{enumerate}
Now to prove $\chi(G)\leq \omega(G)+1$, from \ref{q-col3}, it is enough to prove that $q_1+q_1'+1\leq \omega(G)$.
Note that since $q_1'\geq 1$, we have $Z_1\neq \es$. If $q_1'=1$, then  $Y_1\cup \{v_1,v_2\}$ is a clique (by \ref{Yiclq}) of size $q_1+q_1'+1$ in $G$, and if $q_1'\geq 2$, then since $\{v_1\}$ is complete to $Y_1\cup Z_1$, from \ref{Yiclq} and \ref{YiBigcompZcom},   $\{v_1\}\cup Y_1\cup Z_1$ contains a clique of size  $q_1+q_1'+1$ in $G$. So $q_1+q_1'+1\leq \omega(G)$, and hence  $\chi (G)\leq  \omega(G)+1$, and we are done. So we may assume that either $q_1\leq q_3$ or $q_1'\leq q_3'$.  \end{proof}

		\begin{claim}\label{F2-clqsize}
			$G$ contains a clique of size $q+q'+2$. Also  if   both $B_3$ and $B_5$ are nonempty, then $G$ contains a clique of size $q+3$.
		\end{claim}
		\begin{proof}Recall that for $j,k\in \{1,3\}$ and $j\neq k$, $\{v_{j+2},v_{j-2}\}\cup Z_j$ is complete to $Y_k$ (by \ref{Yiclq}), and that $Y_k$ is a clique (by \ref{Yiclq}). So
$Y_1\cup Z_3\cup \{v_1,v_5\}$ contains a clique of size $q_1+q_3'+2$, and $Y_3\cup Z_1\cup \{v_3,v_4\}$ contains a clique of size $q_3+q_1'+2$.   Now if $q=q_1=q_3$, then   clearly $G$ contains a clique of size $q+q'+2$. So  we have either $q_1> q_3$ or $q_3>q_1$. If $q_1>q_3$, then from \cref{F2-q13-q'13}, we have $q_1'\leq q_3'$, and then  $Y_1\cup Z_3\cup \{v_1,v_5\}$ contains a clique of size $q+q'+2$. Likewise, if $q_3>q_1$, then $Y_3\cup Z_1\cup \{v_3,v_4\}$ contains a clique of size $q+q'+2$. This proves the first assertion of \cref{F2-clqsize}.

To prove the second assertion, recall that from \ref{Yiclq}, $\{v_1,v_5\}\cup B_5$  is complete to $Y_1$ and  $\{v_3,v_4\}\cup B_3$  is complete to $Y_3$. So either $\{v_1,v_5\}\cup B_5\cup Y_1$  or $\{v_3,v_4\}\cup B_3\cup Y_3$ contains a clique of size  $q+3$, and we are done.  This proves  \cref{F2-clqsize}.
\end{proof}
		
		\medskip
	Now from \ref{XiYiancom}, \ref{ABTst}, \ref{Yiclq} and \ref{YiYi+2ancom}, $G[A_2\cup Y_1\cup Y_3]$ is the union of perfect graphs, and hence perfect; so we have $\chi(G[A_2\cup Y_1\cup Y_3]) = \max\{\omega(G[A_2]), \omega(G[Y_1]),\omega(G[Y_3])\}= q$. Also  from \ref{ZiZi+1ancom}, \ref{Ziper} and \cref{F2-X2}, we have $\chi(G[X_2\cup \{v_2\}\cup Z_1\cup Z_3]) = \max\{\omega(G[X_2\cup \{v_2\}]), \omega(G[Z_1]),\omega(G[Z_3])\}= \max\{1,q_1',q_3'\}\leq \max\{1, q'\}$. Now to prove the lemma, we first suppose that   both $B_3$ and $B_5$ are nonempty. Then from \cref{F12-Y245} and \cref{F2-A45emp}, we have  $\chi(G)\leq \chi(G[A_2\cup Y_1\cup Y_3])+\chi(G[X_2\cup \{v_2\}\cup Z_1\cup Z_3])+\chi(G[A_1\cup B_5\cup T\cup \{v_3\}])+\chi(G[A_3\cup B_3\cup \{v_5\}])+\chi(G[\{v_1,v_4\}])\leq q+\max\{1,q'\}+1+1+1$, by \ref{ABTst}. So $\chi(G)\leq (q+3)+1$ or $\chi(G)\leq (q+q'+2)+1$. In both cases, we have $\chi(G)\leq \omega(G)+1$, by \cref{F2-clqsize}, and we are done.  So by using symmetry, we may assume that $B_3=\es$. Then from \cref{F12-Y245} and \cref{F2-A45emp}, we have $\chi(G)\leq \chi(G[A_2\cup Y_1\cup Y_3])+\chi(G[Z_1\cup Z_3])+\chi(G[X_2\cup \{v_2,v_5\}])+\chi(G[A_1\cup B_5\cup T\cup \{v_3\}])+\chi(G[A_3\cup  \{v_1,v_4\}])\leq q+q'+1+1+1\leq \omega(G)+1$, by \ref{ABTst}, \cref{F2-X2} and \cref{F2-clqsize}. This completes the proof of \cref{lem:F12}.
	\end{proofthm}

\setcounter{claim}{0}

	\begin{theorem}\label{thm:F1}
		Let $G$ be a ($P_2+P_3$, gem)-free graph. If $G$ contains an $F_1$, then   either $G\in \cal C$  or $\chi(G)\leq \omega(G)+1$.
	\end{theorem}
	\begin{proofthm}
		Let $G$ be a ($P_2+P_3$, gem)-free graph. From \cref{lem:F11} and \cref{lem:F12}, we may assume that $G$ is ($F_1'$, $F_1''$)-free.
  Suppose that $G$ contains an $F_1$ with vertices and edges as shown in Figure~\ref{fig-F123}. Let $C:=\{v_1,v_2,v_3,v_4,v_5\}$. Then, with respect to $C$, we define the sets $A$, $B$, $X$, $Y$, $Z$ and $T$ as in Section~\ref{genprop}, and we use the properties in Section~\ref{genprop}.  Clearly $y_1^*\in Y_1$ so that $Y_1$ is nonempty. Then $B_2\cup B_4\cup X_3\cup X_4=\es$ (by
		\ref{YiBi+2emp}) and $Z_2\cup Z_5=\es$ (by \ref{ZiYiBi+1oneemp}). Also since $G$ is $(F_1',F_1'')$-free, we have $Y\sm Y_1=\es$ (by \ref{Yiclq}). Next we claim the following:		
		\begin{claim}\label{F3-X25}
			$X_2$ is a stable set. Likewise, $X_5$ is a stable set.
		\end{claim}
		\begin{proof} The proof is similar to the proof of \cref{lem:F12}:~\cref{F2-X2} and so we omit the details. \end{proof}

		\begin{claim}\label{F3-AB}
			At least one of $A_1\cup B_1\cup B_5$ or $A_3\cup A_4$ is empty.
		\end{claim}
		\begin{proof}
		Suppose to the contrary that there are vertices, say $u\in A_1\cup B_1\cup B_5$ and  $w\in A_3\cup A_4$. Then $uw, y_1^*w\in E(G)$ (by \ref{Ai+2Yicom}) and $y_1^*u\in E(G)$ (by \ref{Yiclq}), and then $\{w,u,v_1,v_{2},y_1^*\}$ or $\{w,u,v_1,v_{5},y_1^*\}$  induces a gem which is a contradiction. So \cref{F3-AB} holds. \end{proof}

		\begin{claim}\label{F3-BiXiZi}
		For each $i\in \{1,2,3,4,5\}$ and $i$ mod $5$, the following hold:  $B_{i-1}$ is complete to $X_{i+1}\cup Z_i$. Further if $X_{i+1}\cup Z_i$ is nonempty, then $B_{i-1}$ is complete to $Z_{i+2}$. Likewise,  $B_{i}$ is complete to $X_{i-1}\cup Z_i$. Further if   $X_{i-1}\cup Z_i$ is nonempty, then  $B_{i}$ is complete to $Z_{i-2}$.
		\end{claim}
\begin{proof}For any $b\in B_{i-1}$ and $z\in X_{i+1}\cup Z_i$, since $\{b,v_{i-1},z,v_{i+2},v_{i+1}\}$ does not induce a $P_2+P_3$,
we see that $B_{i-1}$ is complete to $X_{i+1}\cup Z_i$. Next if there are nonadjacent vertices, say $b\in B_{i-1}$ and $z\in Z_{i+2}$, then for any  $z'\in X_{i+1}\cup Z_i$, we see that $\{z,v_{i-1},b,z',v_i\}$ induces a gem (by \ref{ZiZi+1ancom} and by the first assertion). So \cref{F3-BiXiZi} holds. \end{proof}

\medskip
Recall that $A_2\cup A_5\cup X_1$ is anticomplete to $Y_1\cup \{v_1\}$, by \ref{XiYiancom}, and that $Y_1\cup \{v_1\}$ is a clique (by \ref{Yiclq}). To proceed further, we let:
\begin{center}
 \begin{tabular}{l }
  $S_1:=A_2\cup A_5\cup X_1$ and $S_2:=X_2\cup X_5\cup Z_1\cup Z_3\cup Z_4$,\\
    $q:=\max\{\omega(G[S_1]),\omega(G[Y_1\cup \{v_1\}])\}$  and  $q':=\omega(G[S_2])$.
 \end{tabular}
 \end{center}
 	  Clearly $S_1$ is anticomplete to $\{v_1,v_3,v_4\}$ and is not anticomplete to $\{v_2,v_5\}$, and $S_2$ is complete to $\{v_1\}$ and is not anticomplete to $\{v_3,v_4\}$. Note that   $q\geq 2$  and by using \ref{XiYiancom} and \ref{XiP3free}, we have $\chi(G[X_1\cup Y_1\cup \{v_1,v_3,v_4\}])=\max\{\omega(G[X_1]), |Y_1\cup \{v_1\}|, 2\}\leq   \max\{\omega(G[X_1]), |Y_1\cup \{v_1\}|\} \leq q$.  Also $G[S_2]$ is $P_4$-free, and hence  $\chi(G[S_2])=\omega(G[S_2])= q'$.
Moreover, we claim the following:

		\begin{claim}\label{F3-maxclq}
			$G$ contains a clique of size $q+q'$.
		\end{claim}
	\begin{proof}
		First suppose that $q= \omega(G[Y_1\cup \{v_1\}])$. If $q'=1$, then $Y_1\cup \{v_1,v_2\}$ is a clique of size $q+1$, and we are done. So we may assume that $q'\geq 2$, and let $K$ be a clique in $G[S_2]$ such that $|K|=q'$. It is enough to show that $K$ is complete to $Y_1$.
Suppose not. Then there are vertices, say $a,b \in K$ and $y\in Y_1$ such that  $ay\notin E(G)$. Then from \ref{Yiclq}, $a\notin Z_3\cup Z_4$.  If $a\in Z_1$, then  $b\in Z_1$ (by \ref{ZiZi+1ancom}), and then we have a contradiction to \ref{YiBigcompZcom}; so  $a\in  X_2\cup X_5$.
  Then since $\{a,b,y,v_2,v_1\}$ or $\{a,b,y,v_5,v_1\}$ does not induce a gem, we have $by\notin E(G)$, and hence as earlier  we have $b\in X_2\cup X_5$. But then $\{a,b,v_2,y,v_5\}$ induces a $P_2+P_3$ which is a contradiction; so $K$ is complete to $Y_1$. Thus $K\cup Y_1\cup \{v_1\}$ is a clique of size $q+q'$, and we are done.

		So  we may assume that $\omega(G[Y_1\cup \{v_1\}])\leq q-1$ and hence $q= \omega(G[S_1])\geq 3$. Let $Q$ be a clique $G[S_1]$ such that $|Q|=q$.   Now if there are vertices, say $a\in Q\cap A_2$ and $a'\in Q\cap A_5$, then since $|Q|\geq 3$ and from \ref{ABTst}, we have $Q\cap X_1\neq \es$, and then for any $x\in Q\cap X_1$, we see that $\{v_2,a,a',v_5,x\}$ induces a gem; so at least one of $Q\cap A_2$ and $Q\cap A_5$ is empty, and we may assume (up to symmetry) that $Q\cap A_5=\es$. Now if $q'=1$, then $Q\cup \{v_1\}$ is a clique of size $q+1$, and we are done. So we may assume that $q'\geq 2$, and let $Q'$ be a   clique in $G[S_2]$ such that $|Q'|=q'$. We will show that $Q$ is complete to $Q'$. Suppose to the contrary that there are nonadjacent vertices, say $x_1\in Q$ and $z_1\in Q'$. Since $|Q|\geq 3$, there are   vertices in $Q\sm \{x_1\}$, say $x_2$ and $x_3$.  Also since $|Q'|\geq 2$, there is a vertex in $Q'\sm \{z_1\}$, say $z_2$. 		Note that $|Q\cap X_1|\geq |Q|-1\geq 2$,  by \ref{ABTst}. Moreover we have following observations.
		
		\begin{enumerate} [label=($\roman*$)]\itemsep=0pt
			\item\label{F3-maxclq-A2Z13} Since for any $a\in A_2$ and $z\in Z_1\cup Z_3$, $\{a,v_2,z,v_4,v_5\}$ does not induce a $P_2+P_3$, we have $A_2$ is complete to $Z_1\cup Z_3$.
				
			\item\label{F3-maxclq-X2Z34} Since $|Q\cap X_1|\geq 2$,  $Q\cap X_1$ is complete to $Z_3\cup Z_4$, by \ref{Xi+2BigcompZcom}.
			
			\item\label{F3-maxclq-QZ3} From \ref{F3-maxclq-A2Z13} and \ref{F3-maxclq-X2Z34}, we conclude that $Q$ is complete to $Q'\cap Z_3$.
			
			\item\label{F3-maxclq-QZ4} If $z_1\in Z_4$, then from \ref{F3-maxclq-X2Z34},  we have $x_1\in A_2$, and then  $x_2,x_3\in X_1$ (by \ref{ABTst}), and then $\{x_1,x_2,z_1,v_1,v_2\}$ induces a gem (by \ref{Xi+2BigcompZcom}); so $z_1\notin Z_4$. This argument also implies that $Q$ is complete to $Q'\cap Z_4$.

\item\label{F3-maxclq-edg} Since $\{x_1,x_2,v_1,z_1,v_3\}$ or $\{x_1,x_2,v_1,z_1,v_4\}$ does not induce a $P_2+P_3$, we have $x_2z_1\in E(G)$. Likewise, $x_3z_1\in E(G)$, and $z_2$ is adjacent to at least one of $x_2$ and $x_3$.
		\end{enumerate}
		
		From \ref{F3-maxclq-edg}, we may assume that $x_2z_1, x_2z_2\in E(G)$. Now since $\{z_1,z_2,x_1,v_2,x_2\}$ does not induce a gem, we have either $\{z_1,z_2\}$ is not anticomplete to $\{v_2\}$ or $z_2x_1\notin E(G)$. If $\{z_1,z_2\}$ is not anticomplete to $\{v_2\}$, then from \ref{F3-maxclq-QZ4}, we have $z_2v_2\in E(G)$ and $z_2\in Z_4$; so from \ref{ZiZi+1ancom}, $z_1\in X_2$, and then $\{z_1,v_1,v_2,x_1,z_2\}$ induces a gem (by \ref{F3-maxclq-QZ4})  which is a contradiction. So we may assume that $\{z_1,z_2\}$ is anticomplete to $\{v_2\}$ and $z_2x_1\notin E(G)$. Thus from \ref{F3-maxclq-QZ3} and \ref{F3-maxclq-QZ4}, we have $z_1,z_2\notin Z_3\cup Z_4$. Now since $\{z_1,z_2,v_2,x_1,v_5\}$ does not induce a $P_2+P_3$ (by \ref{F3-maxclq-A2Z13}), from \ref{ZiZi+1ancom}, we have  $z_1,z_2\in X_2\cup X_5$, and then from  \cref{F3-X25}, we see that $\{x_1,v_2,z_1,z_2,v_4\}$ induces a $P_2+P_3$ which is a contradiction. This proves that $Q$ is complete to $Q'$.
		So \cref{F3-maxclq} holds. \end{proof}

		\begin{claim}\label{F3-B15}
			We may assume that $B_1$ is an empty set. Likewise, we may assume that $B_5$ is an empty set.
		\end{claim}
	\begin{proof}Suppose  that $B_1\neq \es$, and let $b_1\in B_1$. Then $A\cup B_5\cup Z_3=\es$ (by \ref{BnonempAemp}, \ref{ZiYiBi+1oneemp} and \cref{F3-AB}). We will show that $\chi(G)\leq \omega(G)+1$ as follows:

  First suppose that $\omega(G[X_1])>\omega(G[Y_1\cup \{v_1\}])$; so $\omega(G[X_1])\geq 3$. Then by using \ref{ABTst} and \ref{Yiclq},
		we see that $\chi(G[B_1\cup Y_1\cup\{v_1\}])\leq\chi(G[Y_1\cup \{v_1\}])+1=\omega(G[Y_1\cup \{v_1\}])+1\leq \omega(G[X_1])$, and then since $X_1 \cup \{v_3,v_4\}$ is anticomplete to $B_1\cup Y_1\cup \{v_1\}$ (by \ref{XiYiancom}), using \ref{XiP3free}, we see that $\chi(G[(B_1\cup Y_1\cup X_1\cup \{v_1,v_3,v_4\})])\leq \max\{\chi(G[B_1\cup Y_1\cup \{v_1\}]), \omega(G[X_1]), \chi(G[\{v_3,v_4\}])\}\leq \max\{ \omega(G[X_1]), 2\} \leq q$.  Hence $\chi(G)\leq  \chi(G[(B_1\cup Y_1\cup X_1\cup \{v_1,v_3,v_4\})]) +\chi(G[X_2\cup X_5\cup Z_1\cup Z_4])+\chi(G[B_3\cup T\cup \{v_2,v_5\}])\leq q+q'+1\leq \omega(G)+1$, and we are done. 		
	
So we may assume that $\omega(G[X_1])\leq \omega(G[Y_1\cup \{v_1\}])= |Y_1\cup \{v_1\}|$.   So  $\chi(G)\leq \chi(G[B_1\cup  X_2])+\chi(G[B_3\cup T \cup \{v_2,v_5\}])+\chi(G[X_1\cup Y_1\cup \{v_1,v_3,v_4\}])+\chi(G[X_5\cup Z_1\cup Z_4])\leq 1+1+|Y_1\cup \{v_1\}|+\omega(G[X_5\cup Z_1\cup Z_4])$,
	by \ref{XiYiancom}, \ref{ABTst} and \cref{F3-X25}. Thus to prove that $\chi(G)\leq \omega(G)+1$, it is enough to prove that $G$ contains a clique of size $\omega(G[X_5\cup Z_1\cup Z_4])+|Y_1\cup \{v_1\}|+1$.  Now if $X_5\cup Z_1\neq \es$, then from \cref{F3-BiXiZi} and \ref{Yiclq},  for any $y\in Y_1$ and  $z'\in X_5\cup Z_1$,   since $\{z',b_1,y,v_5,v_1\}$ does not induce a gem, we have $Y_1$ is complete to $X_5\cup Z_1$, and hence  $X_5\cup Z_1\cup Z_4 \cup Y_1\cup \{v_1,b_1\}$ contains a clique of size $\omega(G[X_5\cup Z_1\cup Z_4])+|Y_1\cup \{v_1\}|+1$; so we may assume that $X_5\cup Z_1= \es$. Then from \ref{Yiclq}, we see that $\{v_1,v_2\}\cup Y_1\cup Z_4$ contains a clique of size $\omega(G[Z_4])+|Y_1\cup \{v_1\}|+1$, and we are done. Hence we may assume that $B_1$ is an empty set. Likewise, we may assume that $B_5$ is an empty set.
		 \end{proof}
		
		\begin{claim}\label{F3-A1B3}
			We may assume that one of $A_1$  and $B_3$ is  empty.
		\end{claim}
		\begin{proof}Suppose that $A_1$ and $B_3$ are nonempty. Then from \ref{BnonempAemp}, we have $A\sm A_1=\es$ and $B\sm B_3=\es$. Now for any $a,a'\in A_1$, since $\{v_3,v_4,a,v_1,a'\}$ does not induce a $P_2+P_3$ (by \ref{ABTst}), we have $|A_1|=1$. A similar proof shows that  $|B_3|=1$. We let $A_1:=\{a_1\}$ and $B_3:=\{b_3\}$. Then we have the following observations:
\begin{enumerate} [label=($\roman*$)] \itemsep=0pt
\item\label{F3-A1B3-1} From \ref{Yiclq}, clearly $Y_1$ is complete to $\{a_1\}$, and then for any $y\in Y_1$, since $\{b_3,a_1,v_1,v_2,y'\}$ does not induce a gem (by \ref{Ai+2Yicom}), we see that $Y_1$ is anticomplete to $\{b_3\}$.
\item\label{F3-A1B3-2} For any $x\in X_1$, since  $\{b_3,v_3,x,v_5,v_1\}$ or $\{y_1^*,v_1,x,b_3,v_3\}$ does not induce  a $P_2+P_3$ (by \ref{XiYiancom} and \ref{F3-A1B3-1}), we have   $X_1=\es$.
\item\label{F3-A1B3-3} For any $z\in Z_3$ and $z'\in Z_4$, since  $\{z,v_3,v_4,z',b_3\}$ does not induce a gem (by \ref{ZiZi+1ancom} and \cref{F3-BiXiZi}), we see that at least one of $Z_3$ and $Z_4$ is empty.
    \end{enumerate}
  From \ref{F3-A1B3-3} and by using symmetry, we may assume that $Z_4=\es$. Now  $\chi(G[Y_1\cup \{v_1\}\cup \{b_3,v_4\}]) \leq \max\{|Y_1\cup \{v_1\}|, 2\} \leq q$, by \ref{F3-A1B3-1} and \ref{Yiclq}. Also since $\{v_2\}$ is anticomplete to $S_2$, $\chi(G[X_2\cup X_5\cup Z_1\cup Z_3\cup \{v_2\}])\leq \max\{q',1\}$.  So from \ref{ABTst} and \ref{F3-A1B3-2}, we conclude that $\chi(G)\leq \chi(G[Y_1\cup \{v_1\}\cup \{b_3,v_4\}])+ \chi(G[X_2\cup X_5\cup Z_1\cup Z_3\cup \{v_2\}])+\chi(G[T\cup \{a_1,v_3,v_5\}] \leq q+ \max\{q',1\}+1$. Then from \cref{F3-maxclq} and since $Y_1\cup \{v_1,v_2\}$ is clique of size $q+1$, we have $\chi(G)\leq \omega(G)+1$. So we may assume that one of $A_1$ and $B_3$ is empty. \end{proof}

\medskip
		
		Now from \ref{ABTst}, \cref{F3-AB} and \cref{F3-A1B3}, we have $\chi(G[A_1\cup A_3\cup A_4\cup B_3\cup \{v_2,v_5\}])= 1$.  Thus from \cref{F3-B15}, we conclude that  $\chi(G)\leq \chi(G[X_1\cup Y_1\cup \{v_1,v_3,v_4\}])+\chi(G[S_2])+\chi(G[A_1\cup A_3\cup A_4\cup B_3\cup \{v_2,v_5\}])\leq q +q'+1\leq \omega(G)+1$, by \cref{F3-maxclq}. This completes the proof. 	
	\end{proofthm}

\subsection{($P_2+P_3$, gem, $F_1$)-free graphs that contain an $F_2$}

\setcounter{claim}{0}

	\begin{theorem}\label{thm:F2}
		Let $G$ be a $(P_2+P_3$, gem, $F_1)$-free graph. If $G$ contains an $F_2$, then $\chi(G)\leq \omega(G)+1$.
	\end{theorem}
	\begin{proofthm}
		Let $G$ be a $(P_2+P_3$, gem, $F_1)$-free graph. Suppose that $G$ contains an $F_2$ with vertices and edges as shown in Figure~\ref{fig-F123}. Let $C:=\{v_1,v_2,v_3,v_4,v_5\}$. Then, with respect to $C$, we define the sets $A$, $B$, $X$, $Y$, $Z$ and $T$ as in Section~\ref{genprop}, and we use the properties in Section~\ref{genprop}.  Clearly $z_1,z_1^*\in Z_1$, and hence $Z_1$ is nonempty and $\omega(G[Z_1])\geq 2$; so $\omega(G)\geq 4$. Then $B_2\cup B_4=\es$ (by
		 \ref{ZiYiBi+1oneemp}) and $Z_2\cup Z_5=\es$ (by \ref{ZiZi+1ancom}). Also since $G$ is $F_1$-free, from \cref{F3claim}, we have $Y=\es$, and $X_i$ is a stable set for   each $i\in \{1,2,3,4,5\}$. To proceed further, we let $S:=X_2\cup X_5\cup Z_1\cup Z_3\cup Z_4$. Moreover we claim the following.

		\begin{claim}\label{F4-S-ac}
			The sets $X_2\cup X_5$, $Z_1$, $Z_3\cup \{v_5\}$, and $Z_4\cup \{v_2\}$ are mutually anticomplete to each other.
		\end{claim}
\begin{proof} For any $x\in X_5$ and $z_3\in Z_3$, since $\{z_1,z_1^*,x,z_3,v_5\}$ does not induce a $P_2+P_3$ (by \ref{ZiZi+1ancom}), we see that $X_5$ is anticomplete to $Z_3$. Likewise,  $X_2$ is anticomplete to $Z_4$. Now the proof follows from  \ref{ZiZi+1ancom}. \end{proof}
		
		\begin{claim}\label{F4-S-col}
			We have $2\leq \chi(G[S])\leq \omega(G)-2$. Further  if $Z_4$ is a stable set, $\chi(G[S\cup \{v_2\}])\leq \omega(G)-2$.  Likewise, if $Z_3$ is a stable set, $\chi(G[S\cup \{v_5\}])\leq \omega(G)-2$.
		\end{claim}
\begin{proof} Since $\omega(G[Z_1])\geq 2$, we have  $\chi(G[S])\geq 2$. Also from \ref{Ziper} and  \cref{F4-S-ac}, we see that $\chi(G[S])\leq \max\{\chi(G[X_2\cup X_5]), \chi(G[Z_1]), \chi(G[Z_3]), \chi(G[Z_4])\} \leq \max\{2, \omega(G)-2\} \leq \omega(G)-2$. Further, if $Z_4$ is a stable set, then $\chi(G[S\cup \{v_2\}])\leq \max\{\chi(G[X_2\cup X_5]), \chi(G[Z_1]), \chi(G[Z_3]), \chi(G[Z_4\cup \{v_2\}])\} \leq \max\{2, \omega(G)-2\} \leq \omega(G)-2$.
\end{proof}
		
\medskip
		 Now since $X_1$ is a stable set, from \ref{XiAiAi2col}, we have $\chi(G[A_1\cup A_2\cup A_5\cup B_1\cup B_5\cup X_1\cup\{v_1,v_3,v_4\}])\leq 2$. So if  $X_3\cup X_4=\es$, then  from \ref{ABTst} and  \cref{F4-S-col} we conclude that $\chi(G)\leq\chi(G[S])+\chi(G[A_1\cup A_2\cup A_5\cup B_1\cup B_5\cup X_1\cup\{v_1,v_3,v_4\}])+ \chi(G[A_3\cup A_4\cup B_3\cup \{v_2,v_5\}])\leq  (\omega(G)-2)+2+1\leq \omega(G)+1$, and we are done.  Hence we may assume that $X_3\cup X_4\neq \es$.
By using symmetry, we may assume that $X_3\neq \es$, and let $x_3\in X_3$. Then $B_1=\es$ (by \cref{F3claim}:$\ref{F3fr-XB}$).  Next we claim the following:
		
		\begin{claim}\label{F4-X3Z1}
			$X_3$ is complete to $\{z_1,z_1^*\}$. Likewise, $X_4$ is complete to $\{z_1,z_1^*\}$.
		\end{claim}
		\begin{proof}Suppose to the contrary that there is a vertex in $X_3$, say $x$, such that $\{x\}$ is not complete to $\{z_1,z_1^*\}$. Then since $\{v_1,z_1,v_4,x,z_1^*\}$ does not induce a gem, we have $xz_1,xz_1^*\notin E(G)$, and then $\{v_1,v_2,x,v_4,z_1,z_1^*\}$ induces an $F_1$ which is a contradiction. This proves \cref{F4-X3Z1}. \end{proof}

		\begin{claim}\label{F4-B15-A25}
			$A_5\cup B_5$ is an empty set. Likewise, if $X_4\neq \es$, then $A_2=\es$.
		\end{claim}
		\begin{proof}Suppose to the contrary that there is a vertex, say $u\in A_5\cup B_5$. Then since   $\{u,v_5,x_3,v_2,v_3\}$ and $\{u,v_5,z_1,v_2,v_3\}$  do  not induce  $P_2+P_3$'s, we have $ux_3, uz_1\in E(G)$, and then   $\{x_3,z_1,v_1,v_5,u\}$  or $\{u,x_3,v_4,v_3,z_1\}$ induces a gem (by \cref{F4-X3Z1}) which is a contradiction.  This proves \cref{F4-B15-A25}.  \end{proof}

		\begin{claim}\label{F4-Z4}
			$Z_4$  and $A_2\cup X_1\cup X_3\cup \{v_1,v_3\}$ are stable sets. Likewise, $X_1\cup X_4\cup \{v_1,v_4\}$ is a stable set.
		\end{claim}
	\begin{proof}
 If   there are adjacent vertices, say $u,w\in Z_4$,   then from \cref{F4-X3Z1} and \ref{ZiZi+1ancom}, $\{z,z',x_3,z_1,v_3\}$ induces a $P_2+P_3$; so $Z_4$ is a stable set.
 Next  if  there are adjacent vertices, say $a\in A_2\cup X_1$ and $b\in X_3$, then since $\{a,b,v_4,v_3,z_1,z_1^*\}$ does not induce a gem, we have $az_1,az_1^*\notin E(G)$ (by \cref{F4-X3Z1}), and then $\{a,v_2,z_1,v_4,v_5\}$ or $\{z_1,z_1^*,v_2,a,v_5\}$ induces a $P_2+P_3$; so $A_2\cup X_1$ is anticomplete to $X_3$. Hence for any $r\in A_2$ and $s\in X_1$, since $\{r,s,x_3,v_4,v_3\}$ does not induce a $P_2+P_3$, we have $A_2$ is anticomplete to $X_1$. Now since $X_1$ and $X_3$ are stable sets, the proof follows from \ref{ABTst}.  \end{proof}
		
		 \medskip
		Now if $X_4=\es$, then from \ref{ABTst}, \cref{F4-S-col}, \cref{F4-B15-A25}  and \cref{F4-Z4}, we conclude that $\chi(G)\leq \chi(G[S\cup \{v_2\}])+\chi(G[A_2\cup X_1\cup X_3\cup \{v_1,v_3\}])+\chi(G[A_3\cup A_4\cup B_3\cup T\cup \{v_5\}])+\chi(G[A_1\cup \{v_4\}])\leq (\omega(G)-2)+1+1+1= \omega(G)+1$, and we are done. So we may assume that $X_4$ is nonempty, and let $x_4\in X_4$. Then $A_2=\es$ (by \cref{F4-B15-A25}). Also we claim that:

		\begin{claim}\label{F4A3X3st}
			$A_3\cup A_4\cup B_3\cup X_3\cup\{v_5\}$ is a stable set.
		\end{claim}
	\begin{proof}Suppose  to the contrary that there are adjacent vertices in $A_3\cup A_4\cup B_3\cup X_3$, say $a$ and $b$. Then from \ref{XiYiancom} and \ref{ABTst}, we may assume that $a\in A_4$ and $b\in X_3$. Then since $\{a,b,z_1,v_3,v_4\}$ does not induce a gem (by \cref{F4-X3Z1}), we have $az_1\in E(G)$, and then from \ref{XiYiancom} and \cref{F4-X3Z1}, $\{a,v_4,v_3,x_4,z_1\}$ induces a gem which is a contradiction. So $A_3\cup A_4\cup B_3\cup X_3$ is a stable set. This proves \cref{F4A3X3st}. \end{proof}

\medskip
	Thus from \ref{ABTst}, \cref{F4-S-col}, \cref{F4-B15-A25},  \cref{F4-Z4} and \cref{F4A3X3st}, we have $\chi(G)\leq \chi(G[S\cup \{v_2\}])+\chi(G[X_1\cup X_4\cup \{v_1,v_4\}])+\chi(G[A_3\cup A_4\cup B_3\cup X_3\cup \{v_5\}])+\chi(G[A_1\cup \{v_3\}])\leq (\omega(G)-2)+1+1+1= \omega(G)+1$. This completes the proof.
	\end{proofthm}

\subsection{($P_2+P_3$, gem, $F_1, F_2$)-free graphs that contain an $F_3$}

\setcounter{claim}{0}

	\begin{theorem}\label{thm:F3}
		Let $G$ be a $(P_2+P_3$, gem, $F_1,F_2)$-free graph. If $G$ contains an $F_3$, then $\chi(G)\leq 5 \leq \omega(G)+1$.
	\end{theorem}
	\begin{proofthm}
		Let $G$ be a $(P_2+P_3$, gem, $F_1,F_2)$-free graph. Suppose that $G$ contains an $F_3$ with vertices and edges as shown in Figure~\ref{fig-F123}. Let $C:=\{v_1,v_2,v_3,v_4,v_5\}$. Then, with respect to $C$, we define the sets $A$, $B$, $X$, $Y$, $Z$ and $T$ as in Section~\ref{genprop}, and we use the properties in Section~\ref{genprop}.  Clearly $z_1^*\in Z_1$ and $b_3^*\in B_3$, so that $Z_1$ and $B_3$ are nonempty. Then $(A\sm A_1)\cup B_2\cup B_4=\es$ (by
		\ref{BnonempAemp}) and $Z_2\cup Z_5=\es$ (by \ref{ZiYiBi+1oneemp}).  Moreover since $G$ is $F_1$-free, from \cref{F3claim}, we have $Y=\es$, $X_i$ is a stable set for     $i\in \{1,2,3,4,5\}$, and $X_2\cup X_5=\es$. Also note that $\omega(G)\geq 4$.  Next we claim the following:

\begin{claim}\label{F5-XiZi}
			 For each $i\in \{1,2,3,4,5\}$,  $Z_i$ is a stable set.
		\end{claim}
\begin{proof}
If there are adjacent  vertices, say $z,z'\in Z_i$, then $C\cup \{z,z'\}$ induces an $F_2$; so $Z_i$ is a stable set.
\end{proof}
				
		\begin{claim}\label{F5Z3Z4oneemp}
			At least one of $Z_3$ and $Z_4$ is empty.
		\end{claim}
		\begin{proof}Suppose to the contrary that there are vertices, say $u\in Z_3$ and $w\in Z_4$. Then since $\{b_3^*,v_4,u,v_1,v_2\}$ does not induce a $P_2+P_3$, we have $ub_3^*\in E(G)$. Likewise, $wb_3^*\in E(G)$. But then  $\{u,v_3,v_4,w,b_3^*\}$ induces a gem (by \ref{ZiZi+1ancom}) which is a contradiction. So \cref{F5Z3Z4oneemp} holds. \end{proof}
		
\medskip
		From \cref{F5Z3Z4oneemp} and by using symmetry, we may assume that $Z_4=\es$. Then from \ref{XiYiancom}, \ref{ABTst}, \ref{BnonempAemp} and \ref{ZiZi+1ancom}, and from above claims, we conclude that $\chi(G)\leq \chi(G[A_1\cup B_1\cup B_5\cup X_1\cup \{v_4\}])+\chi(G[X_3\cup \{v_3,v_5\}])+\chi(G[X_4\cup \{v_1\}])+\chi(G[B_3\cup T])+\chi(G[Z_1\cup Z_3\cup \{v_2\}])\leq 1+1+1+1+1= 5 \leq \omega(G)+1$.	This completes the proof of \cref{thm:F3}.
	\end{proofthm}

\subsection{Proof of \cref{structhm}} Let $G$ be a ($P_2+P_3$, gem)-free graph. We may assume that $G$ is not perfect.
 So from the Strong perfect graph theorem \cite{spgt}, $G$ contains either an odd-hole or the complement graph of an odd-hole. Now since any odd-hole of length at least 7 contains a $P_2+P_3$, and since the complement graph of any odd-hole of length at least 7 contains a gem,  we may assume that $G$ contains a $C_5~ (\cong \overline{C_5})$. Then from \cref{thm:C5-F123-free} and from \cref{thm:p2p3K4-bnd}, we may assume that
 $G$ contains at least one of $F_1$, $F_2$ and $F_3$. Now   if $G$ contains an $F_1$, then the theorem follows from \cref{thm:F1};  if $G$ is $F_1$-free and contains an $F_2$, then the theorem follows from \cref{thm:F2}; and finally   if $G$ is ($F_1, F_2$)-free and contains an $F_3$, then the theorem follows from \cref{thm:F3}. This completes the proof of \cref{structhm}. \hfill{$\Box$}

\end{document}